\documentclass[12pt,oneside]{amsart}
\usepackage[utf8]{inputenc}
\usepackage[T1]{fontenc}
\usepackage[BCOR=0mm]{typearea}
\usepackage{amsmath,amsfonts,amssymb,amsthm}
\usepackage{bbm}
\usepackage{graphicx}
\usepackage{stmaryrd}
\usepackage{enumerate}
\usepackage{mathtools}
\mathtoolsset{showonlyrefs,showmanualtags}
\usepackage[font=footnotesize]{caption}
\usepackage[usenames,dvipsnames,svgnames]{xcolor}
\usepackage{stackrel}
\usepackage{hyperref}
\usepackage{esint}
\usepackage{relsize}
\usepackage[foot]{amsaddr}

\def\XXint#1#2#3{{\setbox0=\hbox{$#1{#2#3}{\int}$ }
\vcenter{\hbox{$#2#3$ }}\kern-.59\wd0}}

\newcommand{\grad}{\nabla}
\renewcommand{\div}{\grad\cdot}
\newcommand{\laplace}{\Delta}

\newcommand{\R}{\mathbf{R}}

\newcommand{\uv}{u^{\nu}}

\newcommand{\om}{\omega}
\newcommand{\omv}{\omega^{\nu}}
\renewcommand{\div}{\grad\cdot}
\newcommand{\curl}{\nabla \times}

\newcommand{\phiv}{\phi^{\nu}}

\DeclareMathOperator{\id}{id}

\def\opt{{\mathrm{opt}}}
\newcommand{\zo}{\zeta_{\opt}}
\newcommand{\po}{\pi_{\opt}}

\DeclareMathOperator{\spt}{spt}

\newcommand{\D}{\ensuremath{\mathcal{D}}}

\newcommand{\eps}{\varepsilon}
\newcommand{\e}{\varepsilon}

\newtheorem{theorem}{Theorem}[section]
\newtheorem{lemma}[theorem]{Lemma}

\newtheorem{definition}[theorem]{Definition}

\newtheorem{corollary}[theorem]{Corollary}
\newtheorem{prop}[theorem]{Proposition}

\newcommand{\Ua}{G^{1}_{\eps}}
\newcommand{\Ub}{G^{2}_{\eps}}
\newcommand{\Dd}{\mathcal{D}_{\delta}}

\renewcommand{\L}{\mathcal{L}}

\definecolor{darkblue}{rgb}{0,0,0.6}
\hypersetup{
    pdftitle={Continuity equations with singular coefficients},
    pdfauthor={ Crippa, Nobili, Seis and Spirito},
    colorlinks=true,    % false: boxed links; true: colored links
    linkcolor=darkblue, % color of internal links (red)
    citecolor=darkblue, % color of links to bibliography (green)
    filecolor=darkblue, % color of file links (magenta)
    urlcolor=darkblue   % color of external links (cyan)
}

\numberwithin{equation}{section}


\begin{document}

\title[Continuity and Euler equations with $L^1$ vorticity]{Eulerian and Lagrangian solutions to the continuity and Euler equations with $L^1$ vorticity}
\author{Gianluca Crippa, Camilla Nobili}
\address[GC, CN]{Departement Mathematik und Informatik, Universit\"at Basel, Switzerland. \emph{Email addresses: }\rm \texttt{gianluca.crippa@unibas.ch, camilla.nobili@unibas.ch}.}
\author{Christian Seis}
\address[CS]{Institut f\"ur Angewandte Mathematik, Universit\"at Bonn, Germany. \rm \emph{Email address:} \texttt{seis@iam.uni-bonn.de}.}
\author{Stefano Spirito}
\address[SS]{DISIM - Dipartimento di Ingegneria e Scienze dell'Informazione e Matematica,
		Universit\'a degli Studi dell'Aquila, L'Aquila, Italy. \rm  \emph{Email address:} \texttt{stefano.spirito@univaq.it}.}
%\email{gianluca.crippa@unibas.ch}\email{camilla.nobili@unibas.ch}\email{seis@iam.uni-bonn.de}\email{stefano.spirito@univaq.it}

\date{\today}

\begin{abstract}
	In the first part of this paper we establish a uniqueness result for continuity equations with velocity
	field whose derivative can be represented by a singular integral operator of an $L^1$ function, extending the Lagrangian theory in \cite{BouchutCrippa13}. The proof is based on a combination of a stability estimate via optimal transport
	techniques developed in~\cite{Seis16a} and some tools from harmonic analysis introduced in \cite{BouchutCrippa13}. In the second part of the paper, we address a question that arose in \cite{FilhoMazzucatoNussenzveig06}, namely whether 2D Euler solutions obtained via vanishing viscosity are renormalized (in the sense of DiPerna and Lions)
	when the initial data has low integrability. We show that this is the case even when the initial vorticity is only in~$L^1$, extending the proof for the $L^p$ case in \cite{CrippaSpirito15}.
	%In the present work, we study the linear continuity equation and the 2D Euler equations in the case when the velocity field has a merely integrable curl. We develop a well-posedness theory for the linear equation and derive renormalization properties of viscosity solutions to the Euler equation.
\end{abstract}

\maketitle

%\tableofcontents

\section{Introduction}

In the present work we discuss the equivalence of the Eulerian and the Lagrangian descriptions for solutions
to some equations of fluid dynamics with velocity field with a certain weak regularity. To be more specific, we study the continuity equation and the 2D Euler equations in the case when the velocity field has a merely integrable curl (i.e., the
vorticity of the fluid is $L^1$, but not better). We develop a well-posedness theory for the linear continuity equation 
and derive renormalization properties for solutions to the Euler equation in vorticity form obtained as vanishing 
viscosity limits. %Such properties have some relevance in the Kraichnan--Batchelor (KB) theory of 2D turbulence.
%fluid motions as described by the 2D Euler equation with $L^1$ vorticity, and consider the uniqueness of the corresponding linear continuity equations in $\R^n$ with low regularity assumptions on the velocity fields. We furthermore study renormalization properties of the respective solutions. The investigation of such properties of Euler solutions is related to
%the Kraichnan--Batchelor (KB) theory of 2D turbulence. At the same time, the analysis of linear transport equations  is of independent mathematical interest.
 
Before formulating the precise questions that we address in this paper and motivating the related background from physics, let us review some basic features of the linear theory. The continuity equation 
describes the transport of a conserved quantity $\rho$ by a velocity field $u$. Given an initial configuration $\rho_0$, the Cauchy problem takes the simple form
\begin{equation}\label{CE}
	\left\{
	\begin{array}{rccl}
		\partial_t\rho+\nabla\cdot (u\rho) \!\!\! & = \!\!\!  & 0      & \mbox{in }(0,T)\times \R^n , \\
		\rho(0,\cdot)                     \!\!\!  & =  \!\!\! & \rho_0 & \mbox{in }\R^n.              
	\end{array}			\right.								
\end{equation}
In the classical case of smooth velocity fields and data, the problem of well-posedness is typically solved using the method of characteristics: The unique solution is transported by the flow associated to the velocity field. Since this perspective describes the solution with respect to Langrangian coordinates, we will accordingly refer to it as a \emph{Lagrangian solution}.

Out of the smooth setting there are different ways to give meaning to the continuity equation \eqref{CE}. Whenever the velocity field is regular enough so that a (suitably generalized) flow is well-defined, Lagrangian solutions are reasonable to  be considered. A standard alternative notion which rather takes the partial differential equations (PDE) point of view is that of \emph{distributional (or Eulerian) solutions}. These, however, are well-defined only as long as the product term $u\rho$ is locally integrable. In their seminal paper~\cite{DiPernaLions89}, DiPerna and Lions introduce a new notion of generalized solutions, the so-called \emph{renormalized solutions}, which give sense to \eqref{CE} even if both $u$ and $\rho$ are merely integrable. Roughly speaking, one requires that
\begin{equation}\label{e:rininf}
	\partial_t \beta(\rho) + \div \big( u \beta(\rho) \big) = \big( \beta(\rho) - \rho \beta'(\rho) \big) \div u
\end{equation}
for any smooth function $\beta : \R \to \R$ satisfying suitable growth conditions. Notice that~\eqref{e:rininf} can be 
{\em formally} derived from~\eqref{CE} by applying the chain rule, and that~\eqref{e:rininf} makes sense even when 
one cannot define distributionally the product $u\rho$, due to the low integrability of the two factors.
%
%a renormalized solution is  a distributional solution to which the chain rule applies. % It is immediately verified that such solutions are unique and stable under approximation. Moreover, 
In the case of divergence-free velocity fields, $\div u=0$, these solutions preserve any $L^q$ norm, i.e.,
\begin{equation}
	\label{i3}
	\|\rho(t)\|_{L^q} = \|\rho_0\|_{L^q}
\end{equation}
for any $t\ge0$, whenever the right-hand side is finite.

DiPerna and Lions's theory \cite{DiPernaLions89} in fact shows that these three concepts of solution coincide if 
$u\in L^1(W^{1,p})$ with $\div u\in L^1(L^{\infty})$ and $\rho\in L^{\infty}(L^q)$, where $\frac{1}{p}+\frac{1}{q}=1$. Furthermore, Lagrangian and renormalized solutions still agree 
even if we do not assume any integrability on $\rho$. In either case the Cauchy problem for the continuity equation \eqref{CE}  is well-posed. Ambrosio \cite{Ambrosio04} later generalized the  theory to velocity fields in $L^1(BV)$ and solutions in $L^\infty(L^\infty)$. The precise definition of Lagrangian, distributional and renormalized solutions will be recalled in Section \ref{sec:pre} below. For a review of the DiPerna--Lions theory and its more recent developments we refer to the lecture notes \cite{AmbrosioCrippa14}. Here and at some later occurrences, for notational convenience, we write $L^r(X) = L^r((0,T);X)$ for a function space $X$ on $\R^n$, and if $X = L^r(\R^n)$ we simply write $L^r = L^r((0,T)\times \R^n)$.

Our first main result in this paper concerns a theory for the continuity equation with weakly differentiable velocity fields that fall out of the DiPerna--Lions class $L^1(W^{1,1})$. To be more specific, we consider velocity fields $u$ whose gradient is a singular integral of an $L^1$ function, i.e., $\grad u = K\ast  \omega$ for some singular integral kernel $K$. Typical examples are two- or three-dimensional velocity fields whose curl, which is the vorticity in the context of fluid dynamics, is merely integrable, i.e.,
\begin{equation}\label{i4}
	\omega = \curl u \in L^1.
\end{equation}
In these cases, $K$ is the gradient of the Biot--Savart kernel; see e.g~\cite[Ch.~2]{MajdaBertozzi02}.
Because Calder\'on--Zygmund maximal regularity estimates just fail in $L^1$, in general $\grad u$ does not belong to $ L^1$ but only to $L^1(L^{1,\infty})$. In this regard, the following result extends the theory  in \cite{DiPernaLions89}.

\begin{theorem}\label{Thm1}
	There exists exactly one distributional solution in the class $L^{\infty}(L^{\infty}\cap L^1)$ to the continuity equation with velocity field $u$ with bounded divergence and satisfying $\grad u = K\ast \omega$ for some $\omega\in L^1$. This solution is also a Lagrangian solution and a renormalized solution.
	Also the converse statement holds true: Every Lagrangian or renormalized solution in the class $L^{\infty}(L^1\cap L^{\infty})$ is also a distributional solution.
				
	If in addition $u$ is divergence-free, then there exists a unique renormalized solution in the class $L^{\infty}(L^0)$, which is also a Lagrangian solution. Conversely, every Lagrangian solution in the class $L^{\infty}(L^0)$ is also a renormalized solution.
\end{theorem} 

Here, $L^0=L^0(\R^n)$ denotes the set of all measurable functions $\rho$ on $\R^n$ with values in $\bar \R$ such that $\L^n(\{|\rho|>\lambda\})$ is finite for every $\lambda>0$.

A precise list of assumptions on the singular integral kernel $K$ will be given in the introduction of Section~\ref{sec:pre} below. 

Existence and uniqueness of Lagrangian solutions in the setting of our paper were established earlier in \cite{BouchutCrippa13}, along with  a full theory for the associated ordinary differential equation. 
The nature of the approach of~\cite{BouchutCrippa13} does not allow, however, the
treatment of distributional or renormalized solutions. The major problem in the analysis of distributional
solutions in the setting of~\cite{BouchutCrippa13} (and of the present paper) is the failure of a suitable adaptation 
of the method developed by DiPerna and Lions. To be more specific---for the convenience of the experts among the readers---, it is not clear how a commutator estimate could be established. 

Indeed, instead of following \cite{DiPernaLions89}, the authors of \cite{BouchutCrippa13} exploited the approach
introduced earlier in~\cite{CrippaDeLellis08}. This work provides {\em quantitative} stability, compactness, and 
regularity estimates for Lagrangian flows associated to velocity fields in $L^1(W^{1,p})$ with $p>1$. By using 
more sophisticate harmonic analysis tools, the authors of~\cite{BouchutCrippa13} managed to extend this 
approach to the case $p=1$, and to the case when the gradient of the velocity field is a singular integral of
an integrable function. See also~\cite{Jabin10,HaurayLeBris11,Bohun-2016-bis} for some further extensions 
of this approach.

A PDE analogue of \cite{CrippaDeLellis08,Jabin10} is only very recent. In \cite{Seis16a} a new quantitative 
theory is provided for distributional solutions of the continuity equation in the DiPerna--Lions setting.
This new theory is based on stability estimates for logarithmic Kantorovich--Rubinstein distances, variants of which were introduced earlier in \cite{BOS,OSS,Seis13b}. In the case of velocity fields in $ L^1(W^{1,p})$ with $p>1$, the new stability estimates are optimal \cite{Seis16b} and allow for sharp error estimates for numerical schemes \cite{SchlichtingSeis17,SchlichtingSeis16}.
Let us also mention, in this connection, that quantitative compactness results have been recently
derived in~\cite{BreschJabin15} by a smart technique involving the propagation of suitable
``logarithmic regularity norms'' weighted by solutions of the adjoint equations with a suitable penalization 
term. The authors apply this to get new existence results for the compressible Navier--Stokes equations.

The present work combines the techniques developed in \cite{Seis16a} with certain harmonic analysis 
tools and a new estimate for the difference quotients of the velocity field established in \cite{BouchutCrippa13}.
We will review some tools from \cite{BouchutCrippa13} and \cite{Seis16a} in Sections \ref{sec:pre} and \ref{sec:ot}.

The drawback of the approach in \cite{Seis16a} is that it only applies to distributional solutions and that it does not allow
for a source term on the right hand side of the equation (see for instance~\cite{ColomboCrippaSpirito15} for the
study of the equation with a source with low integrability). As a consequence, the development of a full renormalization
theory in our context requires new ideas. Our strategy is able to handle renormalized solutions only for divergence-free velocity fields, which causes the restriction in the second statement of Theorem \ref{Thm1}.

It turns out that we can use Theorem \ref{Thm1} in the context of the 2D Euler equations with $L^1$ vorticity. Notice that, if $u$ is a two-dimensional divergence-free velocity field described by the Euler equations  and $\omega =\curl u$ the vorticity, then $\omega$ solves the (nonlinear) vorticity equation 
\begin{equation}
	\label{i5}
	\partial_t \omega + u\cdot \grad \omega = 0,
\end{equation}
which can be brought in the conservation form \eqref{CE}. It is clear that the linear theory does not entail uniqueness for the nonlinear problem. Moreover, because $\omega$ is not necessarily bounded, distributional solutions of \eqref{i5} are in general not defined, and in any case Theorem \ref{Thm1} \emph{does not} imply that every $L^1$ distributional solution of the vorticity equation \eqref{i5} is a renormalized or a Lagrangian solution. Combined
with the duality approach developed in~\cite{CrippaSpirito15}, our theory, however, applies to certain particular solutions, namely those which are obtained as the zero-viscosity limit of the Navier--Stokes equations:
%% It is well known that there are two descriptions of the motion of a fluid: the Lagrangian description and the Eulerian one. In a smooth context, say $u\in C(0,T;H^{s}(\R^{2})$, $s>2$\footnote{\red{C: Wht $s>2$? Reference?}}, it is also well-known that the two descriptions coincide, namely Lagrangian solutions are Eulerian and vice versa. An interesting question is to establish whether the two descriptions coincide also in a non smooth setting. 
%%We recall that if we have a weak solution of Euler equations with vorticity in $L^{\infty}(L^{p})$ with $p>2$ then these solutions are also Lagrangian. This was remarked in \cite{FilhoMazzucatoNussenzveig06} and follows basically by the fact that every weak solution of the continuity equations in $L^{\infty}(L^{p})$ is also renormalized if the velocity field has bounded divergence and is $L^{1}(W^{1,q})$ with $q=p/(p-1)$. In the case of of $p<2$, we don't know if this is still true, but we can say that the weak solutions constructed as a limit of exact\footnote{\red{C: What are ``exact'' solutions?}}solutions of Euler equations with vorticity in $L^{\infty}(L^{p})$, $1\leq p<2$, are Lagrangian, see \cite{FilhoMazzucatoNussenzveig06}\footnote{\red{C: Where do they consider Lagrangian solutions? I can't find it:-( I also think they only treat the $p\ge 2$ case...}} and \cite{BohunBouchutCrippa16} for the $p=1$ case. A different,  very natural and physically meaningful approximation is the vanishing viscosity limit: 
We call $\omega$ a \emph{viscosity solution} to the Euler equations \eqref{i5} if 
\[
	\omega = \lim_{\nu \downarrow 0} \, \omega^{\nu},
\]
where $\omega^{\nu}$ is the curl of  some divergence-free velocity field $u^{\nu}$ and (uniquely) solves   the Navier--Stokes vorticity equation with viscosity $\nu$, i.e.,
\[
	\partial_t \omega^{\nu} + u^{\nu}\cdot \grad \omega^{\nu}  = \nu\laplace\omega^{\nu}.
\]
%In \cite{CrippaSpirito15} is proved that viscosity solutions of \eqref{i5} with vorticity in $L^{\infty}(0,T;L^{p}(\R^{2}))$ with $1<p<2$ are Lagrangian. In this paper, by exploiting the uniqueness result of Theorem \ref{Thm1}, we address exactly the $p=1$ case. 
Our result is the following:
\begin{theorem}
	\label{Thm2}
	For initial vorticities in $L^1$, viscosity solutions to the Euler vorticity equations are renormalized solutions and also Lagrangian solutions.
\end{theorem}

This extends to the borderline case $p=1$ the analysis of~\cite{CrippaSpirito15} for the case $p>1$. 
More details will be given in Theorem \ref{teo:4} below.

The fact that viscosity solutions are Lagrangian solutions shows the equivalence between the Eulerian and the Lagrangian description of fluid dynamics---at least in this physically meaningful approximation: As in the smooth setting, the theorem implies that the vorticity is constant along the flow. Existence result for the 2D Euler equations with non-smooth initial vorticity are proved in \cite{Yudovich63,DPM1,Vecchi-Wu,Delort91}.

We want to point out that Theorem \ref{Thm2} is also relevant in connection with the theory of 2D turbulence. 
The phenomenological theory developed by Kraichnan \cite{Kraichnan67} and Batchelor \cite{Batchelor69} is modelled after Kolmogorov's celebrated ``K41'' theory of 3D turbulence. In analogy to the energy cascade in K41, there is the enstrophy cascade picture at the heart of the KB theory. The enstrophy, which is half the integral of the square of vorticity, is a conserved quantity for 2D ideal fluids described by the Euler equations, and it is dissipated by viscous fluids described by the Navier--Stokes equations. In the cascade picture, the nonlinearity transports enstrophy from large to small scales until it is dissipated by viscosity. A key assumption in turbulence theory is that the enstrophy dissipation rate is bounded away from zero uniformly in the viscosity.

Under certain assumptions, this picture, however, is ruled out by the following argument.
%rigorous mathematical analysis:
% To be more specific, we consider the
% vorticity formulation of the 2D Navier--Stokes equations, that is,
%\begin{equation}
%\label{i1}
%\left\{\begin{array}{rll}
% \partial_t \omega_{\nu} + u_{\nu}\cdot \grad\omega_{\nu} - \nu\laplace \omega_{\nu} &\!\!\!= 0&\mbox{in }(0,T)\times \R^2, \\ 
% u_{\nu}&\!\!\!= K\ast \omega_{\nu} &\mbox{in }(0,T)\times \R^2, \\ \omega_{\nu} (0,\cdot) &\!\!\!=  \bar \omega &\mbox{in }\R^2, \end{array}\right.
%\end{equation}
%where $K$ is the 2D Biot--Savart kernel, i.e.,
%\begin{equation}
%\label{i2}
%K(z) = \frac1{2\pi}\frac{z^{\perp}}{|z|^2}.
%\end{equation}
%Hence $\omega_{\nu}$ is the vorticity that corresponds to the divergence-free velocity field $u_{\nu}$, i.e., $\omega_{\nu} = \curl u_{\nu}$.  
It is easily checked that the Navier--Stokes equations dissipate the enstrophy  $\frac12\|\omega^{\nu}(t)\|_{L^2}^2$ at the  rate $\nu \|\grad \omega^{\nu}(t)\|_{L^2}^2$. If the latter was bounded away from zero by a positive constant $C$, then
\[
	\|\omega^{\nu}(t)\|_{L^2}^2 + Ct \le \| \omega_0\|_{L^2}^2,
\]
for any $t>0$. In order to perform the limit $\nu\to 0$, it remains to invoke a standard compactness argument. We find a  function $\omega$ in $L^{\infty}(L^2)$ which satisfies the 
Euler equations in vorticity form
%, i.e.,
%\begin{equation}
%\label{19}
%\left\{\begin{array}{rll}
% \partial_t \omega + u\cdot \grad\omega  &\!\!\!= 0&\mbox{in }(0,T)\times \R^2, \\ 
% u&\!\!\!= K\ast \omega &\mbox{in }(0,T)\times \R^2, \\ \omega (0,\cdot) &\!\!\!=  \bar \omega &\mbox{in }\R^2, \end{array}\right.
%\end{equation}
and such that  $\|\omega(t)\|_{L^2} < \| \omega_0\|_{L^2}$ for any positive $t$. That means that the limiting Euler equations do no preserve enstrophy. This, however, contradicts the DiPerna--Lions theory of renormalized solutions \cite{DiPernaLions89}. Indeed, because  $ \|\grad u\|_{L^2}= \|\omega\|_{L^2} $, the advecting velocity field is in the DiPerna--Lions class, and thus $\omega$ is a renormalized solution,
which entails \eqref{i3} with $\rho=\omega$ and $q=2$.

It is natural to ask if such dissipation is in fact present under more general assumptions. For a given Banach space $X$, the questions are thus the following: \emph{Given an initial datum in $X$, is there a viscosity solution to the Euler vorticity equation?} And, if yes: \emph{Is that viscosity solution a renormalized solution?}  These questions are mathematically interesting independently from their fluid dynamical background. 

Among some other spaces, this questions were studied for $L^p$ spaces in \cite{FilhoMazzucatoNussenzveig06} ($p\ge 2$) and   \cite{CrippaSpirito15} ($1<p<2$), and in either case both questions (when applicable) are answered positively. Notice that for $p<4/3$, a priori estimates available for $u$ and $\omega$ are not enough to guarantee that the nonlinear term $u\omega$ is in $L^1$. For this reason, in order to make sense of \eqref{CE}, solutions to the Euler equation are defined as  renormalized solutions and the second question is redundant. The arguments in \cite{FilhoMazzucatoNussenzveig06} and \cite{CrippaSpirito15} are hinged on the fact that Calder\'on--Zygmund theory for the Biot--Savart kernel (given implicitly in  \eqref{i4}) yields $\|\grad u\|_{L^p} \lesssim \|\omega\|_{L^p}$ precisely if $p\in(1,\infty)$. In the borderline case $p=1$,
where this estimate fails, $u$ does not have Sobolev regularity and therefore DiPerna--Lions theory is not applicable. See however \cite{Vecchi-Wu}. Notice that the other borderline case $p=\infty$ is on the contrary well-behaved \cite{Yudovich63}: in fact, even uniqueness for the nonlinear problem can be proven. 
Our   Theorem~\ref{Thm2} extends the results from \cite{FilhoMazzucatoNussenzveig06} and \cite{CrippaSpirito15} to the case $p=1$. We build up on the linear theory established in Theorem~\ref{Thm1} and closely follow the argumentation developed in \cite{CrippaSpirito15}. 

The paper is organized as follows: In Section \ref{sec:pre} we recall some basic definitions of solutions to linear continuity equations, some auxiliary results from harmonic analysis  and  interpolation and embedding estimates for weak Lebesgue spaces. Section \ref{sec:ot} contains some preliminaries on optimal transportation distances for specific choices of concave cost functions. In Section \ref{Uni} we prove our uniqueness result for linear continuity equations when the velocity field is a singular integral of an $L^1$ function. The final Section \ref{sec:euler} is devoted to the analysis of vanishing viscosity solutions for the 2D Euler equations.
Throughout the paper we will use the short notation $a\lesssim b$ whenever $a\le C b$ for some constant $C$ depending only on the space dimension $n$ and on other quantities that we do not specify as they do not play any role in the estimates.				
%
%If $u$ denotes the velocity field and $\rho_0$ the initial configuration, then the Cauchy problem reads
%
%On the initial data we assume 
%\begin{equation}\label{initial-data}
%	\rho_0\in L^{\infty}\cap L^{1}(\R^n) 
%\end{equation}
%and throughout the paper we assume our density has the following regularity
%\begin{equation}\label{extra-on-rho}
%	\rho\in L^{\infty}((0,T); L^{\infty}\cap L^1(\R^n))
%\end{equation}
%The Lagrangian formulation of the same problem reads
%\begin{equation}\label{CE2}
%	\begin{array}{lll}
%		\frac{dX}{dt}(t,x) & = & u(t, X(t,x)) \quad t\in (0,T) \\
%		X(0,x)             & = & x\,.                          
%												    
%	\end{array}
%\end{equation}
%\begin{theorem}\label{main-th1}
%\end{theorem}
%\begin{theorem}\label{main-th2}
%\end{theorem}
%%
\section{Linear continuity equations and singular integrals}\label{sec:pre}
%{ \bl {\bf Where do we put this assumptions? In the Intro or here?}
%
% On the initial data we assume 
%\begin{equation}\label{initial-data}
%	\rho_0\in L^{\infty}\cap L^{1}(\R^n) 
%\end{equation}
%and throughout the paper we assume our density has the following regularity
%\begin{equation}\label{extra-on-rho}
%	\rho\in L^{\infty}((0,T); L^{\infty}\cap L^1(\R^n))
%\end{equation}
%}

The present section is divided into three subsections: In the first one, we recall the definitions of distributional, Lagrangian and renormalized solutions to the continuity equation \eqref{CE} under quite general assumptions. In the second subsection, we specify the assumptions on the velocity field and the singular integral kernel, and collect a number of technical results that were previously established in \cite{BouchutCrippa13}. In the last subsection we summarize some inequalities involving weighted Lebesgue spaces that we will need in the following.

\subsection{Distributional, renormalized and Lagrangian solutions to linear continuity equations}
%In this subsection we recall the main definitions and results concerning the linear continuity equations with velocity fields satisfying \eqref{u-1}, \eqref{u-2} and \eqref{u}. 
We start by recalling the usual definition of distributional solutions. 
\begin{definition}[Distributional solutions]\label{def:ds}
	Let $u\in L^{1}((0,T);L_{loc}^{p}(\R^{n}))$ and $\rho_0\in L^q_{loc}(\R^n)$ be given for some $q$ such that $1/p+1/q\leq1$. A function $\rho$ is called a distributional solution of \eqref{CE} if  $\rho\in L^{\infty}((0,T);L_{loc}^{q}(\R^{n}))$  and
	\begin{equation*}
		\iint \rho(\partial_t\phi+u \cdot \nabla\phi)\,dxdt + \int\rho_{0}\phi|_{t=0}\,dx= 0,
	\end{equation*}
	for any $\phi\in C^{\infty}_{c}([0,T)\times\R^{n})$. 
	\end{definition}
							
	Whenever the velocity's divergence is bounded from below, distributional solutions in the sense of the previous definition can be obtained by smooth approximation. This standard argument is performed, for instance, in \cite[Propositon II.1]{DiPernaLions89}.
							
	In the context of linear transport and continuity equations,  DiPerna and Lions \cite{DiPernaLions89} introduced the concept of renormalized solutions.
	\begin{definition}[Renormalized solutions]\label{def:rs}
		Let $u\in L^{1}((0,T);L^{1}_{loc}(\R^{n}))$ be given  with $\div u\in L^{1}((0,T);L^1_{loc}(\R^{n}))$ and $\rho_0\in L^0(\R^n)$. Then,  $\rho\in L^{\infty}([0,T);L^{0}(\R^{n}))$  is a renormalized solution of \eqref{CE} if for any $\beta\in C^1(\R)\cap L^{\infty}(\R)$, 
			$\beta$ vanishing in a neighborhood of $0$ and $|\beta'(s) s|$ bounded,  it holds
			\begin{equation*}
				\iint \beta(\rho)(\partial_t\phi+u\cdot\nabla\phi)+(\div u)\big(\beta'(\rho)\rho-\beta(\rho)\big)\phi\,dxdt + 
				\int \beta(\rho_{0})\phi|_{t=0}\,dx=0,
			\end{equation*}
			for any $\phi\in C^{\infty}_{c}([0,T)\times\R^{n})$. 
			\end{definition}
															
			Note that the definition of renormalized solution makes sense even when it is not possible to define distributional solutions, e.g., if $\rho u\notin L^{1}_{loc}$. In fact, under the hypotheses on $\rho$ and $\beta$, it holds that $\beta(\rho)$ and $\rho\beta'(\rho)$ are both in $L^{\infty}(L^1\cap L^{\infty})$. Moreover, if $\rho$ and $u$ are as in Definition \ref{def:ds} above, an approximation argument shows that renormalized solutions are in fact distributional solution, cf.~\cite[Theorem II.3]{DiPernaLions89}.
															
			Before defining Lagrangian solutions we first need to introduce regular Lagrangian flows.
			\begin{definition}[Regular Lagrangian flows]\label{def:rlf}
				Let $u\in L^{1}((0,T);L^{1}_{loc}(\R^n))$ be given. We say that $X: (0,T)\times\R^n\to\R^n$ is a regular Lagrangian flow associated to $u$ if 
				\begin{enumerate}
					\item[(1)] for a.e.~$x\in \R^n$ the map $t\mapsto X(t,x)$ is an absolutely continuous integral solution of the ordinary differential equation  $\frac{d}{dt}X(t,x)=u(t,X(t,x))$ for $t\in(0,T)$ with $X(0,x)=x$;
					\item[(2)] there exists a constant $L$, called compressibility constant, independent of $t$ such that 
					      \begin{equation}\label{eq:compress}
					      	\mathcal{L}^{n}(B)\leq L\mathcal{L}^n(\{x\in \R^n:\,X(t,x)\in B\})
					      \end{equation} 
					      for any Borel set $B\subset\R^n$. 
				\end{enumerate}
			\end{definition}
																	
			For a given regular Lagrangian flow, we furthermore define the corresponding Jacobian determinant $JX$ by $JX(t,x) := \det(\nabla_x X(t,x))$.	We will call a regular Lagrangian flow \emph{invertible} if 
			$X(t, \cdot)$ is a.e.~invertible for any $t\in (0,T)$. In this case we  denote by $X^{-1}(t,\cdot)$ its inverse map. 				
			Then the definition of  Lagrangian solutions of 
			\eqref{CE} is the following:
															 
			\begin{definition}[Lagrangian solutions]\label{def:ls}
				Let $\rho_0\in L^0(\R^n)$ be given. A function $\rho$ is called a Lagrangian solution of \eqref{CE} if  $ \rho\in L^{\infty}((0,T);L^{0}(\R^n))$  and there exists an invertible regular Lagrangian flow $X$ associated to $u$ such that
				\begin{equation*}
					\rho(t,x)=\frac{\rho_0(X^{-1}(t,x))}{JX(t, X^{-1}(t,x))}
				\end{equation*}
				for all $t\in(0,T)$ and a.e.~$x\in\R^n$.					
			\end{definition}
																							
			Notice that Lagrangian solutions are just those solutions that are obtained in the smooth setting via the method of characteristics. We can more compactly write $\rho(t,\cdot) = X(t,\cdot)_{\#}\rho_0$, where $\#$ denotes the pushforward operator.												
																		 
			\subsection{Velocity fields whose gradient is given by a singular integral}
							
			In this subsection we collect some harmonic analysis tools for singular integrals defined by
			\[
				S \omega : = K\ast \omega
			\]
			for sufficiently fast decaying functions $\omega$. We focus on  integral kernels $K: \R^n\setminus\{0\}\to \R$ which satisfy the following properties:
			\begin{itemize}
				\item[K1)] $K\in\mathcal{S}'(\R^{n})$ and $\widehat{K}\in L^{\infty}(\R^n)$, where $\widehat K$ denotes the Fourier transform of $K$;
				\item[K2)] $K|_{\R^n\setminus \{0\}}\in C^1(\R^n\setminus \{0\})$;
				\item[K3)] there exists a constant $C\geq 0$ such that 
				      $$|K(x)|\leq \frac{C}{|x|^n} \qquad \mbox{ for every } x\neq 0\,;$$
				\item[K4)] there exists a constant $C\geq 0$ such that 
				      $$|\nabla K(x)|\leq \frac{C}{|x|^{n+1}} \qquad \mbox{ for every } x\neq 0\,;$$
				\item[K5)] there exists a constant $C\geq 0$ such that 
				      $$\left|\int_{R_1<|x|<R_2}K(x)\, dx \right|\leq C \qquad \mbox{ for every } 0<R_1<R_2<\infty\,.$$
			\end{itemize}
			Typical examples of admissible kernels are first order  derivatives of the two or three-dimensional Biot-Savart kernels, or, more general, second order derivatives of  Newtonian potentials. For a comprehensive theory of singular integrals we refer to \cite{Stein70}.
						
			By standard Calder\'on--Zygmund theory, $S$ extends to a continuous operator on $L^p$ 
			as long as $p\in (1,\infty)$, and continuity fails if $p=1$. Instead, one has the weak estimate		
			\begin{equation}\label{eq:s3}
			\|S\omega\|_{L^{1,\infty}}\lesssim \|\omega\|_{L^1}.
			\end{equation}
			Recall that, for arbitrary $p$, the space $L^{p,\infty}$ denotes the weak $L^p$ space (or Lorentz space), which is associated to the quasi-norm
			\begin{equation}\label{eq:defm}
				\|f\|_{L^{p,\infty}}^p=\sup_{\lambda>0}\Big\{ \lambda^p\mathcal{L}^{n}\left(\left\{x\in \R^n:|f(x)|>\lambda\right\}\right)\Big\},
			\end{equation}
			for every measurable function $f$ on $\R^n$. Observe that the quantity $\|\cdot\|_{L^{p,\infty}}$ is not a norm, because it lacks the triangular inequality. We also recall that the embedding $L^{p}\subset L^{p,\infty}$ holds with $\|f\|_{L^{p,\infty}}\leq \|f\|_{L^p}$ and that the inclusion is strict for any $p<\infty$. We also adopt the standard convention that $L^{\infty,\infty}=L^{\infty}$.
						
			A central tool in classical Calder\'on--Zygmund theory is the maximal operator $M$, defined by
			\[
				M(f)(x)=\sup_{\e>0}\frac{1}{\mathcal{L}^n(B_{\e}(x))}\int_{B_{\e}(x)}|f(y)|\, dy.
			\]
			This operator is itself continuous from $L^p$ to $L^p$ provided that $1< p \leq \infty$. Again,  continuity ceases to hold at $p=1$. Instead, in analogy to \eqref{eq:s3} one has 
			\[
				\| M ( f)\|_{L^{1,\infty}} \lesssim  \|f\|_{L^1}.
			\]
			Although this weak bound holds for the maximal function and for the singular operator (see \eqref{eq:s3}) separately, we cannot hope that the same bound holds for the composition $M\circ S$. Such an estimate, however, would be essential for an adaptation of the method introduced in \cite{CrippaDeLellis08} (and translated to the PDE setting in \cite{Seis16a}). Indeed, one of the key estimates in \cite{CrippaDeLellis08} is the control  of difference quotients by gradients. In a first step the authors use the fact that difference quotients are bounded by maximal functions,
			\[
				\frac{|u(x)-u(y)|}{|x-y|}\lesssim M(\grad u)(x) + M(\grad u)(y),
			\]
			for a.e.~$x,y$. This estimate is rather elementary and belongs to the class of Morrey estimates; its proof is essentially contained in \cite[pp.~143-144]{EvansGariepy92}. In the second step the authors apply the continuity estimate for maximal function operators, which is suitable only if $p>1$. For gradients of the form $S\omega$ with merely integrable $\omega $, this strategy needs some modifications.									
			As in  \cite{BouchutCrippa13}, we will consider the following smooth variant of the maximal function:
			\[
				M_{\sigma}(f)(x):=\sup_{\varepsilon>0}\left|\frac{1}{\varepsilon^n}\int_{\R^n}\sigma\left(\frac{x-y}{\varepsilon}\right)f(y)\, dy\right|,
			\] 
			where $\sigma\in C_c^{\infty}(\R^n)$. Notice that  the difference with  the classical maximal function is not only the smooth cut-off, but also that the modulus is taken only after the computation of the (smooth) average. It is proved in \cite{BouchutCrippa13}  that for appropriate convolution kernels $\sigma$ the compositions of $S$ with  these smooth maximal functions do satisfy the estimate
			\begin{equation}\label{lem:max}
				||M_{\sigma}(S\omega)||_{L^{1,\infty}(\R^n)}\lesssim ||\omega||_{L^1(\R^n)},
			\end{equation}
			see \cite[Theorem 3.3]{BouchutCrippa13}. Regarding the Morrey-type estimate, it is proven in \cite{BouchutCrippa13} that if $\omega\in L^1(L^1)= L^1((0,T)\times \R^n)$ then there exists a function $G$ on $(0,T)\times \R^n$ and for a.e. $t$ a set $N_t$ with $\mathcal{L}^n(N_t)=0$ such that 
			\begin{equation}
				\label{10}
				\frac{|u(t,x)-u(t,y)|}{|x-y|}\lesssim G(t,x)+G(t,y)\, \qquad \forall x,y\not\in N_t\,.
			\end{equation}	
			For every $\eps>0$, this function can be furthermore decomposed into the sum $G^1_{\eps}+G^2_{\eps}$:
			\begin{equation}
				\label{11}
				\|G^1_{\eps}\|_{L^1(L^{1,\infty})} \le \eps,\quad \|G^2_{\eps}\|_{L^1(L^2)} \le C_{\eps},
			\end{equation}
			where $C_{\eps}$ depends, besides on $\eps$, also on the equi-integrability of $\omega$. This in particular prevents the applicability of this technique to the case when $\omega$ is a measure with a non-trivial singular part.
			%
			%	Moreover, in \cite{BouchutCrippa13}, Lemma \ref{lem:diffincr} and Lemma \ref{lem:max} are used in combination with a decomposition depending on the equi-integrability of the function $g$. We start by recalling that 
			%				given $\eps$, an arbitrary positive number, it is easy to prove that $g\in L^1(\R^n)$ can be decomposed into the sum $g^{(1)} + g^{(2)}$ with
			%				\begin{equation}\label{4}
			%					\|\ga\|_{L^1(\R^n)}\leq \varepsilon \quad\mbox{ and } \quad \|\gb\|_{L^{\infty}(\R^n)}\leq c_{\varepsilon}\,
			%				\end{equation}
			%				for some constant $c_{\eps}<\infty$. Notice that, by interpolation between $L^1$ and $L^{\infty}$, we also have 
			%				$$\|\gb\|_{L^2(\R^n)}\leq C_{\varepsilon}\,.$$
			%				It is convenient to shorten the notation by setting
			%				$$G(x):=M_{\sigma}(Sg)(x)\,.$$
			%				Noticing that $M_{\sigma}(Sg)$ is a sub-linear operator and by using Lemma \ref{lem:max} we have 
			%				\begin{lemma}
			%					For any $\e>0$, given $g_{1}$ and $g_{2}$ as in \eqref{4}, the function $G$ can be decomposed as $G=G^{1}+G^{2}$, with $G^{1}$ and $G^{2}$ such that 
			%					\begin{align}
			%						\label{Uga}\|G^{(1)}\|_{L^{1,\infty}(\R^n)} & \leq C_1\|g^{(1)}\|_{L^1(\R^n)}, \\
			%						\label{Ugb}\|G^{(2)}\|_{L^2(\R^n)}          & \leq C_1\|g^{(2)}\|_{L^2(\R^n)}. 
			%					\end{align}
			%				\end{lemma}
			%			
			\subsection{Some inequalities}
						
			We conclude this section with auxiliary embedding and interpolation inequalities on a finite measure space $(X,\mu)$. We will later need such inequalities in the specific case of measures of the form $d\mu(t,x) = \chi_{(0,T)}(t)|\rho(t,x)| d\L^1\otimes d \L^n$ and similar. For this purpose we define $L^p$ and weak-$L^p$ norms on $(X,\mu)$ by
			$$\|f\|^p_{L^p(\mu)}=\int_{X}|f|^p\, d\mu $$
			and $$\|f\|^p_{L^{p,\infty}(\mu)}=\sup_{\lambda>0}\Big\{\lambda^p\mu\left(\left\{x\in X :  |f(x)|>\lambda\right\}\right)\Big\},$$
			respectively.
													
			\begin{lemma}
				For $1 \leq r < p$ it holds that
				\begin{equation}\label{weak-embed}
					\|f\|^r_{L^r(\mu)}\leq \frac{p}{p-r}\mu(X)^{1-\frac rp}\|f\|^r_{L^{p,\infty}(\mu)}\,.
				\end{equation}
				%					In particular, we have the following embedding
				%					\begin{equation}
				%						(L^p(X,\mu)\subseteq) L^{p,\infty}(X,\mu)\subseteq L^q(X,\mu)\,,
				%					\end{equation}
				%					where $X$ is any space with finite measure.
			\end{lemma}
								
			\begin{proof}
				Let us rewrite the $L^r$ norm of $f$ in terms of the measure of its superlevel sets. Denoting $m(\lambda) = \mu\left(\left\{x\in X:\: |f(x)|\ge \lambda\right\}\right)$, we have
				\begin{align*}
					\|f\|_{L^r(\mu)}^r=\int_0^{\infty}r\lambda^{r-1}m(\lambda)\, d\lambda =\int_0^{\alpha}r\lambda^{r-1}m(\lambda)\, d\lambda+\int_{\alpha}^{+\infty}r\lambda^{r-1}m(\lambda)\, d\lambda\,, 
				\end{align*}
				where $\alpha$ is a positive number that we will choose later. 
														
				The first term is trivially estimated as follows:						$$\int_0^{\alpha}r\lambda^{r-1}m(\lambda) \, d\lambda\leq \mu(X)\alpha^r\,.$$
				We turn to the estimate of the second term. Using the inequality $\lambda^p m(\lambda) \le  \|f\|_{L^{p,\infty}(\mu)}^p$, we find
				\[			\int_{\alpha}^{+\infty}r\lambda^{r-1}m(\lambda)\, d\lambda\leq
					\frac{r}{p-r}\|f\|_{L^{p,\infty}(\mu)}^p\alpha^{r-p}\,.
				\]
				Therefore putting all together we have 
				$$\|f\|_{L^r(\mu)}^r\leq \mu(X)\alpha^r+\frac{r}{p-r}\|f\|_{L^{p,\infty}(\mu)}^p\alpha^{r-p}\,.$$
				Optimizing the right-hand side with respect to  $\alpha$ we find 
				$\alpha = \mu(X)^{-\frac1p}\|f\|_{L^{p,\infty}(\mu)}  $, and thus
				\[
					\|f\|_{L^r(\mu)}^r \le \frac{p}{p-r} \mu(X)^{1-\frac{r}p} \|f\|_{L^{p,\infty}(\mu)}^r.
				\]
				This is the desired inequality.
			\end{proof}	
						
			The following interpolation inequality is a variant of \cite[Lemma 2.2]{BouchutCrippa13}.
						
			\begin{lemma}\label{L10}
				For any $1< p<\infty$ it holds that
				\begin{equation}\label{interpol}
					\|f\|_{L^1(\mu)}\leq \frac{p}{p-1}\|f\|_{L^{1,\infty}(\mu)}\left[1+\log\left(\frac{\mu( X)^{1-\frac{1}{p}}\|f\|_{L^{p,\infty}(\mu)}}{\|f\|_{L^{1,\infty}(\mu)} }\right)\right]\,.
				\end{equation}
			\end{lemma}									
													
			\begin{proof}
				We start again by writing the $L^1$ norm of $f$ in terms of its level sets. Setting as above
				$m(\lambda)=\mu(\{x\in  X: |f(x)|>\lambda\})$, we have
				$$\|f\|_{L^1(\mu)}=\int_0^{\alpha}m(\lambda)\, d\lambda\,+\int_{\alpha}^{\beta}m(\lambda)\, d\lambda\,+\int_{\beta}^{\infty}m(\lambda)\, d\lambda\,,$$
				where 
				\begin{equation}
					\alpha=\frac{\|f\|_{L^{1,\infty}(\mu)}}{\mu( X)}  \qquad \mbox{and}\qquad \beta=\left(\frac{\|f\|^p_{L^{p,\infty}(\mu)}}{\|f\|_{L^{1,\infty}(\mu)}}\right)^{\frac {1}{p-1}}\,.
				\end{equation}
				The choice of $\alpha$ and $\beta$ is admissible in the sense that $\alpha \le \beta$. Indeed, because $m(\lambda)\le \mu( X)$, it holds that
				%					\begin{eqnarray*}
				%						\lambda \mu(\{|f|>\lambda\})
				%						&\leq& \lambda \mu(\{|f|>\lambda\})^{\frac{1}{p}}\mu(\{|f|>\lambda\})^{1-\frac{1}{p}}\\
				%						&\leq& \lambda \mu(\{|f|>\lambda\})^{\frac{1}{p}}\mu( X)^{1-\frac{1}{p}}\\
				%						&\leq& \lambda \mu(\{|f|>\lambda\})^{\frac{1}{p}}\|\rho\|_{L^1}^{1-\frac{1}{p}}\,,
				%					\end{eqnarray*}
				$$\|f\|_{L^{1,\infty}(\mu)}\leq\mu( X)^{1-\frac{1}{p}} \|f\|_{L^{p,\infty}(\mu)}\, $$
				which is equivalent to $\alpha\leq\beta$. 
								
				Using the trivial bound $	m(\lambda)\leq \mu( X)$ again,  we see that		
				$$\int_0^{\alpha}m(\lambda)\, d\lambda\, \leq \alpha\mu( X) = \|f\|_{L^{1,\infty}(\mu)}.$$
				On the one hand, from the estimate
				$\lambda m(\lambda)\leq  \|f\|_{L^{1,\infty}(\mu)}$, we deduce that
				$$\int_{\alpha}^{\beta}m(\lambda)\, d\lambda\,\leq \|f\|_{L^{1,\infty}(\mu)} \log\left(\frac{\beta}{\alpha}\right) =\|f\|_{L^{1,\infty}(\mu)}\log\left(\mu( X)\left(\frac{\|f\|_{L^{p,\infty}(\mu)}}{\|f\|_{L^{1,\infty}(\mu)}}\right)^{\frac{p}{p-1}}\right).$$
				On the other hand, from the estimate $\lambda^p m(\lambda)\leq  \|f\|_{L^{p,\infty}(\mu)}^p$, we have
				$$\int_{\beta}^{\infty}m(\lambda)\, d\lambda\,\leq \frac1{p-1}\|f\|_{L^{p,\infty}(\mu)}^p \beta^{1-p}=\frac{1}{p-1}\|f\|_{L^{1,\infty}(\mu)}\,.$$
				A combination of the previous estimates yields the statement of the lemma.
			\end{proof}
																			
			\section{Optimal transportation with logarithmic cost functions}\label{sec:ot}
										
			In this section, we briefly review some tools from the theory of optimal transportation that will become relevant in our subsequent analysis. For a comprehensive introduction into the topic, we refer to \cite{Villani03}.
										
			We consider two non-negative distributions $\rho_1$ and $\rho_2$ on $\R^n$ with the same total mass
			\begin{equation}\label{total-mass}
				\int \rho_1\, dx=\int \rho_2 \, dx <\infty ,
			\end{equation}
			and denote by $\Pi(\rho_1,\rho_2)$ the set of the corresponding transport plans. Namely, $\pi \in \Pi(\rho_1,\rho_2)$ is a measure on the product space $\R^n\times \R^n$ with marginals $\rho_1$ and $\rho_2$, i.e.,
			\begin{equation*}
				\pi[A\times \R^n]=\int_A \rho_1 \, dx\,,\quad \pi[\R^n\times A]=\int_A\rho_2\, dy\,,
			\end{equation*}
			for all measurable sets $A\subset \R^n$,
			or equivalently
			\begin{equation}\label{marginal}
				\iint (f_1(x)+f_2(y)) d\pi(x,y)=\int f_1\rho_1\, dx+\int f_2\rho_2 \, dy\,,
			\end{equation}
			for all functions $f_1$ in $L^1(\rho_1\, dx)$ and $f_2$ in $L^1(\rho_2\, dx)$.  For a given cost function $c$ on $\R_+$ the minimal transportation cost is defined as
			\begin{equation}\label{transp-cost}
				\D_c(\rho_1,\rho_2)=\inf_{\pi\in \Pi(\rho_1,\rho_2)}\iint c(|x-y|)d\pi(x,y) .
			\end{equation}
			Informally speaking, $\D_c(\rho_1,\rho_2)$ measures the minimal total cost for transferring one configuration $\rho_1$ (e.g., a pile of sand) into another configuration $\rho_2$ (e.g., a hole), if the cost for the transport of a single item over the distance $z$ is given by $c(z)$.			
										
			In this paper we will only consider strictly concave cost functions. Notice that strictly concave cost functions naturally induce a metric on $\R^n$, given by $d(x,y) = c(|x-y|)$. In this case, \eqref{transp-cost} admits the dual formulation  
			\begin{equation}\label{var}
				\D_c(\rho_1,\rho_2)=\sup_{\zeta}\left\{\int \zeta(\rho_1-\rho_2)\, dx : |\zeta(x)-\zeta(y)|\leq d(x,y)\right\}\,.
			\end{equation}
			This identity is a variant of the classical Kantorovich duality of optimal transportation and is usually referred to as the  Kantorovich--Rubinstein theorem, cf.~\cite[Theorem 1.14]{Villani03}. 
			The theorem has an immediate consequence: $\D_c(\rho_1,\rho_2)$ is a transshipment cost that depends only on the difference of $\rho_1$ and $\rho_2$. In particular it extends to densities that are not necessarily nonnegative but satisfy \eqref{total-mass}. Moreover, $\D_c$ defines a metric on the space of densities with the same total mass, cf.~\cite[Theorem 7.3]{Villani03}. This metric is called a Kantorovich--Rubinstein distance. 
			For any function $\rho\in L^1(\R^n)$ with zero average,
			\[
				\int \rho\, dx=0\,,
			\]
			we introduce the norm		
			\[
				\mathcal D_c(\rho):=\mathcal D_c(\rho,0):=\mathcal D_c(\rho^+,\rho^-),
			\]
			where the superscripted plus and minus signs indicate the positive and the negative parts, respectively.
						
			We note that the primal problem \eqref{transp-cost} admits a unique minimizer $\po \in \Pi(\rho^+,\rho^-)$, called optimal transport plan, and the dual problem \eqref{var} admits a (non-unique) maximizer $\zo$, called Kantorovich potential, which are characterized by the identity
			$$\zo(x)-\zo(y)=d(x,y)\qquad \mbox{ for } d\po\mbox{-almost all } (x,y)\,,$$
			cf.~\cite[Theorem 2.45]{Villani03} . It is not difficult to infer from this identity that $\zo$ is weakly differentiable with
			\begin{equation}\label{weak-diff}
				\nabla \zo(x) = \nabla \zo(y) =\nabla_x d(x,y)=c'(|x-y|)\frac{x-y}{|x-y|}\qquad \mbox{ for } d\po\mbox{-almost all } (x,y).
			\end{equation}
			Morover, there exist two maps $S$ and $T$ such that
			\begin{equation}
				\label{2}
				\po = (\id \times T)_{\#}\rho^+ = (S\times \id)_{\#}\rho^-,
			\end{equation}
			and $S$ and $T$ obey the relations $\rho^+ = S_{\#}\rho^-$ and $\rho^- = T_{\#} \rho^+$; cf.~\cite{GangboMcCann96,PegonSantambrogioPiazzoli15}.
										
			In most parts of this paper, we will consider a  smooth variant of the bounded logarithmic cost function introduced in \cite{Seis16a}, namely
			\begin{equation}\label{cost-function}
				c_{\delta}(z)=\log\left(\frac{\tanh(z)}{\delta}+1\right)\,,
			\end{equation}
			and write $\D_{\delta}(\rho)$ as an abbreviation of $\D_{c_{\delta}}(\rho)$ for notational convenience. In the following, $\po$ and $\zo$ will always denote the optimal transport plan and Kantorovich potentials corresponding to this norm. If $d_{\delta}(x,y)$ is analogously defined, we notice that $d_{\delta}(x,y) \leq \delta^{-1}|x-y|$, and thus $\zo$ is a Lipschitz function and by normalizing $\zo(0)=0$ it is bounded by $\log(\delta^{-1}+1)$. For later reference, we notice that \eqref{weak-diff} becomes
			\begin{equation}
				\label{1}
				\grad \zo (x) = \grad\zo(y) =\frac{1-\tanh^2(|x-y|)}{\delta + \tanh(|x-y|)}\frac{x-y}{|x-y|}
			\end{equation}
			for $d\po$-almost all $(x,y)$.														
						
			We finally consider the Kantorovich--Rubinstein norm
			\[
				\D(\rho) : = \inf_{\pi \in \Pi(\rho^+,\rho^-)} \iint \tanh |x-y|\, d\pi(x,y)
			\]
			on the space of functions with zero average.  A control of $\D(\rho)$ by $\D_{\delta}(\rho)$ is established in the following lemma.
								
			\begin{lemma}\label{lemma-extra}
				Let $\rho$ be an average-zero function in $L^1(\R^n)$.
				Then for any $\gamma>0$ and $\delta>0$ it holds that
				\begin{equation}\label{new-dist-est}
					\D(\rho)\leq  \frac{\Dd(\rho)}{\log\frac1\gamma}+\frac{\delta}{\gamma} \|\rho\|_{L^1}.
				\end{equation}
			\end{lemma}
																							
			This is a variant of an estimate first proved in \cite{Seis16a}. For the convenience of the reader, we redo the short proof with the modified distance functions.	
			\begin{proof}
				We define $K=\{(x,y)\in \R^n\times\R^n:\:  c_{\delta}(|x-y|)\leq \log\frac1\gamma)\}$ and denote by $K^c$ it complement. Throughout the proof, $\po$ denotes the optimal transport plan corresponding to $\D_{\delta}(\rho)$. On the one hand, we have 
				$$\iint_{K} \tanh |x-y|\,  d\po\leq \frac{\delta}{\gamma}\po[K]\,$$
				and $\po[K]$ is bounded	by $\po[\R^n\times \R^n] = \|\rho\|_{L^1}$. On the other hand, by the boundedness of the hyperbolic tangent, we estimate
				\[
					\iint_{K^c}\tanh |x-y|\, d\po\leq \po[K^c]\le  \frac1{\log\frac1\gamma} \iint_{K^c} c_{\delta}(|x-y|) \, d\po(x,y) \le \frac{\Dd(\rho)}{\log\frac1\gamma}.
				\]  
				Combining both estimates yields
				\[
					\D(\rho) \le \iint \tanh|x-y|\, d\po(x,y) \le  \frac{\Dd(\rho)}{\log\frac1\gamma}+\frac{\delta}{\gamma} \|\rho\|_{L^1}. \qedhere
				\]
				
			\end{proof}	
																																																																														
			\section{Uniqueness of distributional solutions of the continuity equation}\label{Uni}
										
			In this section, we state and prove our first main result, the well-posedness of the Cauchy problem \eqref{CE} in the sense of distributions introduced in Definition \ref{def:ds}. To specify the assumptions on the velocity field, we assume that 
			\begin{equation}
				\label{u-1} u\in L^{p,\infty}((0,T)\times \R^n)
			\end{equation} 
			for some $p>1$ and that $\grad u = K\ast \omega$ for some $L^1$ function $\omega$, which in components reads
			\begin{equation}
				\label{13}
				\partial_i u_j = \sum_{\ell=1}^L K_{ij}^{\ell}\ast \omega^{\ell}_{ij} \quad \mbox{for some }\omega_{ij}^1,\dots, \omega_{ij}^L\in L^1((0,T);L^1(\R^n)),
			\end{equation}
			for any $i,j\in\{1,\dots,n\}$, where the $K_{ij}^{\ell}$'s satisfy the hypotheses K1) to K5). Moreover we suppose that 
			\begin{equation}\label{u-2}
				\nabla\cdot u\in L^1((0,T);L^{\infty}(\R^n)).                 
			\end{equation}
			We remark that condition \eqref{u-1} substitutes the usual growth condition assumed in the DiPerna--Lions theory. We also impose that the initial datum is integrable and bounded, i.e.,
			\begin{equation}
				\label{16}
				\rho_0\in L^{\infty}\cap L^1(\R^n).
			\end{equation}
						
			Let us now give a precise result.												
																																		  
			\begin{theorem}\label{th1}
				Let $u$ be a velocity field satisfying \eqref{u-1}, \eqref{13} and \eqref{u-2} and let the initial datum $\rho_0$ be such that \eqref{16} holds. Then the Cauchy problem \eqref{CE} has a unique distributional solution $\rho$ in the class $ L^{\infty}((0,T);L^{\infty}\cap L^{1}(\R^n))$.
			\end{theorem}
										
			Notice that distributional solutions are well-defined, because
			\begin{equation}
				\label{20}
				\|u\rho\|_{L^1} \lesssim \|\rho\|_{L^1}^{1-\frac1p} \|\rho\|_{L^{\infty}}^{\frac1p} \|u\|_{L^{p,\infty}}
			\end{equation}
			by the virtue of Lemma \ref{weak-embed}, and the right-hand side is finite by the assumptions on $\rho$ in the theorem and on $u$ in~\eqref{u-1}.
									
			It is worth pointing out that the assumption \eqref{u-2} on the divergence of $u$ is used only to prove existence, but it is not needed for uniqueness.	
									
			Under the  hypotheses \eqref{u-1}--\eqref{16}, (unique) Lagrangian solutions (see Definition \ref{def:ls}) were constructed in \cite{BouchutCrippa13}. These solutions solve \eqref{CE} also in the sense of distributions and in the sense of renormalized solutions (see Definition \ref{def:rs}). The new contribution of Theorem \ref{th1} is thus the uniqueness part. Moreover, combining our result with the ones in \cite{BouchutCrippa13}, we deduce the following corollary.
																
			\begin{corollary}\label{T2}					
				Under the assumptions of Theorem \ref{th1}, the notions of distributional,  Lagrangian and renormalized solutions are equivalent.
			\end{corollary}
										
			Furthermore, it was shown in \cite{BouchutCrippa13} that Lagrangian solutions exist and are renormalized solutions even under the milder assumption that $\rho_0\in L^0$. Under the assumption that $u$ is divergence-free, the composed functions $\beta(\rho)$ also solve the continuity equation in the class $L^{\infty}(L^1)$, and are unique. From Theorem \ref{th1} we thus infer the following consequence.
										
			\begin{corollary}\label{c2}
				Let $u$ be a divergence-free velocity field satisfying \eqref{u-1} and \eqref{13} and let the initial datum $\rho_0$ be in $L^0$. Then there exists a unique renormalized solution to the Cauchy problem \eqref{CE}. Moreover, the notions of Lagrangian and renormalized solutions are equivalent.
			\end{corollary}

			Theorem \ref{th1} and Corollaries \ref{T2} and \ref{c2} contain all statements of Theorem \ref{Thm1} of the introduction.

			To simplify the notation in the following, we will simply write $\rho_t$, $\pi_t$ and $\zeta_t$ for $\rho(t,\cdot)$ $\po(t)$ and $\zo(t,\cdot)$, respectively.

			Our proof of Theorem \ref{th1} combines ideas recently developed in \cite{Seis16a} with the harmonic-analysis techniques from \cite{BouchutCrippa13}.			 
			The main tool is the following ``stability'' estimate, whose proof will be postponed to Subsection \ref{ss1}.							 
			\begin{prop}\label{pr3}
				Let $\rho\in L^{\infty}((0,T);L^{\infty}\cap L^1(\R^n))$ be a non-trivial solution of the continuity equation \eqref{CE} with zero average. Then there exists for every $\eps>0$ a finite constant $C_{\eps}>0$ such that for every $\delta>0$ it holds 
				\begin{equation}\label{stab-est}
					\sup_{0\le t\le T}\Dd(\rho_t)
				\lesssim \Dd(\rho_0)+ \varepsilon \|\rho\|_{L^1}\left[1+\log\left( \frac1{\eps\delta} \left(\frac{\|\rho\|_{L^1}}{\|\rho\|_{L^{\infty}}}\right)^{1-\frac1p}\|u\|_{L^{p,\infty}}\right)\right]+ C_{\varepsilon} \|\rho\|_{L^\infty( L^2)} .                                                                                                              \end{equation}
		\end{prop}  
																															
		The proof of Theorem \ref{th1} follows directly from Proposition \ref{pr3}. 
																															 
		\begin{proof}[Proof of Theorem \ref{th1}]
			Existence of distributional solutions follows immediately from the construction of Lagrangian solutions in \cite[Section 7]{BouchutCrippa13}. In view of the linearity of the problem, uniqueness holds if the trivial solution is the unique solution with $\rho_0=0$. We argue by contradiction and assume that there is a non-trivial solution in $L^{\infty}(L^{\infty}\cap L^1)$  with zero initial datum. Then Proposition \ref{pr3} yields
			$$ \sup_{0\le t\leq T}\Dd(\rho_t)
		\lesssim\varepsilon \|\rho\|_{L^1}\left[1+\log\left( \frac1{\eps\delta} \left(\frac{\|\rho\|_{L^1}}{\|\rho\|_{L^{\infty}}}\right)^{1-\frac1p}\|u\|_{L^{p,\infty}}\right)\right]    + C_{\varepsilon} \|\rho\|_{L^\infty( L^2)} .$$
		Since, by assumption $u$ and $\rho$ are bounded in $L^{p,\infty}$ and $L^{\infty}(L^{\infty}\cap L^{1})$, respectively, we may write
		$$ \sup_{0\le t\leq T}\Dd(\rho_t)
		\lesssim C \varepsilon\left[1+\log\left(\frac{1}{\delta \varepsilon}\right)\right]+
		C_{\varepsilon} \,,$$
		where the constant $C$ depends on $\|\rho\|_{L^{\infty}}$,$\|\rho\|_{L^1}$ and $\|u\|_{L^{p,\infty}}$.
		We let $\theta>0$ be arbitrarily small and we fix a $\varepsilon$ such that
		$$\frac{\varepsilon\left[1+\log\left(\frac{1}{\delta \varepsilon}\right)\right]}{|\log \delta|}\leq \frac{\theta}{2}\quad \mbox{ uniformly  in }\delta\ll1\,.$$
		Notice that this  is possible because
		$$\frac{\varepsilon\left[1+\log\left(\frac{1}{\delta \varepsilon}\right)\right]}{|\log \delta|}=\frac{\varepsilon (1+|\log\varepsilon|+|\log \delta|)}{|\log \delta|}\leq \varepsilon\left(2+|\log\varepsilon|\right)\,,$$
		and the right-hand-side converges to $0$ as $\varepsilon\rightarrow 0$.
		Now that $\varepsilon$ and in particular $C_{\varepsilon}$ are fixed, we choose $\delta$ such that 
		$$\frac{ C_{\varepsilon}}{|\log \delta|}\leq \frac{\theta}{2}\,,$$
		and obtain that $\Dd(\rho_t) \lesssim \theta |\log \delta|$ uniformly in $t$.  Because $\theta$ was arbitrary, the latter implies that
		\begin{equation}\label{r}
			\frac{\Dd(\rho_t)}{|\log\delta|}\rightarrow 0 \mbox { as } \delta \rightarrow 0
		\end{equation}
		for all times $t$.	
				
		It remains to conclude that \eqref{r} implies that $\rho_t=0$ for all $t$, which contradicts the hypothesis at the beginning of the proof. In fact, from Lemma \ref{lemma-extra} with $\gamma = \sqrt\delta$ it follows that
		\[
			\D(\rho_t) \le \sqrt{\delta} \|\rho_t\|_{L^1} + 2\frac{\Dd(\rho_t)}{|\log \delta|}.
		\]
		Letting $\delta\to 0$, we find $\D(\rho_t)= 0$ thanks to \eqref{r}, and thus $\rho_t=0$ because $\D$ is a norm. This concludes the proof.
	\end{proof}
																			
	%																	
	%
	%				\begin{proof}
	%					The existence of a Lagrangian and renormalized solution of \eqref{CE} follows from Theorem \ref{teo:bc}. Suppose that $\rho_1$ and $\rho_2$ are two renormalized solutions. Since $\beta$ vanishes in a neighbourhood of $0$ and $u$ is divergence-free it follows that $B:=\beta(\rho_1)-\beta(\rho_2)\in L^{\infty}((0,T);L^{\infty}\cap L^1(\R^n))$
	%					is a distributional solution of 
	%					\begin{equation*}
	%						\begin{array}{rcc}
	%							\partial_tB+\nabla\cdot (uB) & = & 0\,, \\
	%							B(0,\cdot)                   & = & 0\,. 
	%						\end{array}
	%					\end{equation*}
	%					Then, by Theorem \ref{th1}, we that $B=0$ a.e. in $(0,T)\times\R^{n}$, $\beta(\rho_1)=\beta(\rho_2)$. Then, by varying $\beta$ we conclude that $\rho_1=\rho_2$ a.e. in $(0,T)\times\R^{n}$.
	%				\end{proof}
	%																				

	\subsection{Proof of Proposition \ref{pr3}}\label{ss1}
	In most parts of the proof, we follow \cite{Seis16a}. Starting point is the following rate of change formula for the Kantorovich--Rubinstein norm $\Dd(\rho_t)$, which is valid for distributional solutions to the continuity equation.
																				  
	\begin{lemma}\label{lemma1}
				
		The mapping	$ t \mapsto \Dd(\rho_t)$ is  absolutely continuous with
		\begin{equation}\label{e:deri}
			\frac{d}{dt}\Dd(\rho_t)=\int \nabla \zeta_t\cdot u(t,\cdot) \rho_t\, dx, \qquad \text{for a.e.~$t \in (0,T)$,}
		\end{equation}
		where $\zeta_t$ is the Kantorovich potential corresponding to $\Dd(\rho_t)$. 
	\end{lemma}
						
	The statement of the lemma was already proved in  \cite{Seis16a}. Here, we present a slightly simplified argument for the convenience of the reader.
					
	\begin{proof}
		We first notice that the definition of distributional solutions, Definition \ref{def:ds}, and a standard approximation argument imply that
		\begin{equation}
			\label{14}
			\int \zeta \left(\rho_t - \rho_{t-h}\right) \, dx = \int \grad\zeta \cdot \left(\int_{t-h}^t u(s,\cdot)\rho_s\, ds\right)dx
		\end{equation}
		for all $\zeta \in C_c^{\infty}(\R^n)$ and a.e.~$t\in(0,T)$ and $h\in\R$ such that $t-h\in (0,T)$. Moreover, because $\rho$ and $u\rho$ are both in $L^1(L^1)$, cf.~\eqref{20}, it is enough to consider \eqref{14} for $\zeta$'s in $W^{1,\infty}(\R^n)$. 
		
We first show that the mapping	$ t \mapsto \Dd(\rho_t)$ is absolutely continuous, hence classically differentiable at a.e.~$t \in (0,T)$. By the optimality of $\zeta_t$ at time $t$, it holds that
		\begin{align}\label{e:AC1}
			\Dd(\rho_t)- \Dd(\rho_{t-h}) & \le \int \zeta_t \left(\rho_t-\rho_{t-h}\right)\, dx        \\
			                             & = \int_{t-h}^t \int \grad \zeta_t \cdot u(s, \cdot)\rho_s\, dxds, 
		\end{align}
		for a.e.~$t\in(0,T)$ and $h\in\R$. Analogously, by the optimality of $\zeta_{t-h}$ at time $t-h$, it holds that
		\begin{align}\label{e:AC2}
			\Dd(\rho_t)- \Dd(\rho_{t-h}) & \ge \int \zeta_{t-h} \left(\rho_t-\rho_{t-h}\right)\, dx        \\
			                             & = \int_{t-h}^t \int \grad \zeta_{t-h} \cdot u(s,\cdot)\rho_s\, dxds, 
		\end{align}
		for a.e.~$t\in (0,T)$ and $h\in\R$. Using that $\zeta_t$ is Lipschitz with $\|\nabla\zeta_t\|_{\infty}\leq1/\delta$ for a.e.~$t \in (0,T)$, we can combine~\eqref{e:AC1} and~\eqref{e:AC2} to the effect that
		$$
		\left| \Dd(\rho_t)- \Dd(\rho_{t-h}) \right| \leq \frac{1}{\delta} \left| \int_{t-h}^t \int | u(s, \cdot)\rho_s | \, dxds \right|
		$$
		for a.e.~$t\in (0,T)$ and $h\in\R$. Using again that $u\rho\in L^1(L^1)$ by~\eqref{20}, we conclude that $t \mapsto \Dd(\rho_t)$ is absolutely continuous.
		
		We eventually prove the expression~\eqref{e:deri} for the derivative. To this aim, it is enough to consider again~\eqref{e:AC1}, divide by $h$, and let $h\to 0$. By Lebesgue's differentiation theorem we find
		$$
		\lim_{h \downarrow 0} \frac{\Dd(\rho_t)- \Dd(\rho_{t-h})}{h} \leq \int \nabla \zeta_t\cdot u(t,\cdot) \rho_t\, dx
		$$
		and
		$$
		\lim_{h \uparrow 0} \frac{\Dd(\rho_t)- \Dd(\rho_{t-h})}{h} \geq \int \nabla \zeta_t\cdot u(t,\cdot) \rho_t\, dx,
		$$
		which implies~\eqref{e:deri} at a.e.~$t$.
	\end{proof}

	In the next step, we integrate the identity from Lemma \ref{lemma1} and estimate the right-hand side with the help of 	the explicit formulas we found for $\grad \zeta_t$ on $\spt \pi_t$.				
																													  
	\begin{lemma}\label{L1} It holds that
		\[
			\sup_{0\leq t\le T}  \Dd(\rho_t) \le \Dd(\rho_0) + \int_0^T \iint  \frac{|u(t,x)-u(t,y)|}{\delta +|x-y|}\, d\pi_t(x,y)dt.
		\]
	\end{lemma}
						
	A very similar version of this estimate was first derived in \cite{BOS} by using Lagrangian coordinates and the primal formulation \eqref{transp-cost}. Here, we follow \cite{Seis16a} which is based on the dual formulation \eqref{var} and \eqref{weak-diff}.
	\begin{proof} Using the marginal conditions \eqref{marginal} for the transport plans, we can rewrite the estimate from Lemma \ref{lemma1} as
		\[
			\frac{d}{dt} \Dd(\rho_t) = \iint \left(u(t,x)\cdot \grad\zeta_t(x) - u(t,y)\cdot \grad\zeta_t(y)\,\right) d\pi_t(x,y).
		\]
		This formulation is advantageous because the derivative of $\zo$ is explicitly known on $\spt \po$. Indeed, in view of \eqref{1}, we have
		\[
			\frac{d}{dt}\Dd(\rho_t)  = \iint \frac{1-\tanh^2(|x-y|)}{\delta + \tanh(|x-y|)}\frac{x-y}{|x-y|}\cdot(u(t,x)-u(t,y))\, d\pi_t(x,y).
		\]
		Thanks to the  elementary estimate \[
		0<\frac{1-\tanh^2(z)}{\delta +\tanh(z)} \le \frac1{\delta +z}
	\]
	for any nonnegative $z$, the latter becomes
			
	\[
		\frac{d}{dt}\Dd(\rho_t)
		\le \iint \frac{|u(t,x)-u(t,y)|}{\delta +|x-y|}\, d\pi_t(x,y).
	\]
	Integration in time, the fact that the initial value is attained weakly and the fact that Kantorovich--Rubinstein distances metrize weak convergence, cf.~\cite[Theorem 7.12]{Villani03}, imply the statement of the lemma.
	\end{proof}																												
	At this point, our proof substantially deviates from the one in \cite{Seis16a}, but exploits the techniques developed in \cite{BouchutCrippa13}. We first notice that \eqref{10} and the decomposition $G = G^1_{\eps} + G^2_{\eps}$ allow for estimating the integrand as
	\begin{equation}\label{15}
		\frac{|u(t,x)-u(t,y)|}{\delta + |x-y|}  \lesssim \min\left\{\frac{|u(t,x)|+|u(t,y)|}{\delta}, G^1_{\eps}(t,x) +G^1_{\eps}(t,y)\right\}+ G^2_{\eps}(t,x) + G^2_{\eps}(t,y)
	\end{equation}
	for a.e. $t$ and every $x,y\not\in N_t$, where $\mathcal{L}^n(N_t)=0$. Observe that this estimate holds for a.e. $t$ and $\pi_t$-a.e. $(x,y)\in \R^n\times \R^n$, because the marginals of the measure $\pi_t$ are absolutely continuous w.r.t the Lebesgue measure $\mathcal{L}^n$.
	
	We will estimate the integrals of the terms on the right-hand side separately. For the first term we have the following integral bound:																									  
	\begin{lemma}\label{L2}
		It holds that
		\begin{align*}
			\MoveEqLeft\int_0^T \iint \min\left\{\frac{|u(t,x)|+|u(t,y)|}{\delta}, \Ua(t,x) + \Ua(t,y)\right\}\, d\pi_t(x,y)dt\\
			  & \lesssim  \eps \|\rho\|_{L^{\infty}}\left(1+\log\left(\frac1{\eps\delta}\left(\frac{\|\rho\|_{L^1}}{\|\rho\|_{L^{\infty}}}\right)^{1-\frac1p}\|u\|_{L^{p,\infty}}\right)\right), 
		\end{align*}
		provided that $\rho\not\equiv0$.
	\end{lemma}

	\begin{proof}
		Let us first bound the expression on the left-hand side by the sum $I_1+I_2$ where
		\begin{align*}
			I_1 & = \int_0^T\iint \min\left\{\frac{|u(t,x)|}{\delta}, \Ua(t,x)\right\}\, d\pi_t(x,y)\, dt \\&\quad + \int_0^T\iint \min\left\{\frac{|u(t,y)|}{\delta}, \Ua(t,y)\right\}\, d\pi_t(x,y)\, dt, \\
			I_2 & = \int_0^T\iint \min\left\{\frac{|u(t,x)|}{\delta}, \Ua(t,y)\right\}\, d\pi_t(x,y)\, dt \\&\quad + \int_0^T\iint \min\left\{\frac{|u(t,y)|}{\delta}, \Ua(t,x)\right\}\, d\pi_t(x,y)\, dt\,. 
		\end{align*}
		Thanks to the marginal condition \eqref{marginal}, the diagonal terms in $I_1$ immediately simplify to
		\[
			I_1 = \int_0^T\int \min\left\{\frac{|u|}{\delta},\Ua\right\} |\rho|\, dx\, dt,
		\]
		because $\rho^+ + \rho^- = |\rho|$.
		We write $\psi = \min\{|u|/\delta,\Ua\}$ for notational convenience. The difficulty in estimating $\psi$ in the $L^1$ norm comes from the fact that $\Ua$ is bounded only in the weaker space $L^{1,\infty}$, cf.~\eqref{11}. The term $\delta^{-1}|u|$ on the other hand is controlled even in $L^{p,\infty}$, cf.~\eqref{u-1}, but with a large factor $\delta^{-1}$. Following the strategy developed in \cite{BouchutCrippa13}, we combine the controls in the spaces $L^{1,\infty}$ and $L^{p,\infty}$ with the help of the interpolation inequality \eqref{interpol}. For this, we introduce the finite measure $d\mu(t,x)  = \chi_{(0,T)}(t)|\rho(t,x)|d\L^1\otimes d\L^n$ on $\R^{n+1}$.
		We then have on the one hand that
		\begin{equation}\label{A}
			\|\psi\|_{L^{1,\infty}(\mu)} \le \|\Ua\|_{L^{1,\infty}(\mu)}\le \|\rho\|_{L^{\infty}} \|\Ua\|_{L^1(L^{1,\infty})} \le \eps \|\rho\|_{L^{\infty}},
		\end{equation}
		where in the last inequality we have used \eqref{11}.
		On the other hand, we have
		\begin{equation}\label{B}
			\|\psi\|_{L^{p,\infty}(\mu)} \leq \frac1{\delta} \|u\|_{L^{p,\infty}(\mu)}\le  \frac1{\delta} \|\rho\|_{L^{\infty}}^{\frac1p}\|u\|_{L^{p,\infty}},
		\end{equation}
		and the expression on the right is finite thanks to \eqref{u-1}.												
		Combining these two estimates with the interpolation inequality \eqref{interpol} then yields
				
		\[
			I_1  = \|\psi\|_{L^1(\mu)} \lesssim \eps \|\rho\|_{L^{\infty}}\left(1+\log\left(\frac1{\eps\delta}\left(\frac{\|\rho\|_{L^1}}{\|\rho\|_{L^{\infty}}}\right)^{1-\frac1p}\|u\|_{L^{p,\infty}}\right)\right).
		\]

		The estimate of the off-diagonal terms $I_2$ is quite similar. 
		It makes in addition use of the optimal transport maps introduced in \eqref{2}. It holds that
		\[
			I_2 = \int_0^T\int \min\left\{\frac{|u\circ S|}{\delta},\Ua\right\}\rho^-\, dy\, dt +  \int_0^T\int \min\left\{\frac{|u\circ T|}{\delta},\Ua\right\}\rho^+\, dx\, dt=:I_2^a + I_2^b,
		\]
		where the composition acts in the spatial variable only. The treatment of both terms $I_2^a$ and $I_2^b$ is very similar and it is thus enough to focus on one of them, say $I_2^a$. 
		We set $\psi = \min\{|u\circ S|/\delta, \Ua\}$ and define the finite measure $d\mu(t,x)  = \chi_{(0,T)}(t)\rho^-(t,x)d\L^1\otimes d\L^n$ on $\R^{n+1}$. The estimate in \eqref{A} applies without changes with the new choices of $\psi$ and $\mu$. 
		Also the final estimate in \eqref{B} remains valid; its derivation, however, needs a small modification. In fact, similarly as before, we obtain
		\[
			\|\psi\|_{L^{p,\infty}(\mu)} \le \frac1{\delta} \|u\circ S\|_{L^{p,\infty}(\mu)}.
		\]
		We now use the relation $\rho^+= S_{\#}\rho^-$ to the effect that
		\[
			\mu(\{ |u\circ S|>\lambda\})   =\Big(S_{\#}\rho^- \L^1\otimes \L^n\Big)(\{|u|>\lambda\}) =\Big(\rho^+ \L^1\otimes \L^n\Big)(\{|u|>\lambda\}).
		\]
		From this, the final estimate in \eqref{B} follows. 			
		It remains to argue as for $I_1$ to conclude.
	\end{proof}																											
	The integral estimate of the second term in \eqref{15} is contained in the following lemma.																									
	\begin{lemma}\label{L3}
		It holds that
		\[
			\int_0^T \iint (\Ub(t,x) + \Ub(t,y))\, d\pi_t(x,y) \leq  C_{\eps} \|\rho\|_{L^{\infty}(L^2)} .
		\]
	\end{lemma}
																														 
	\begin{proof}[Proof of Lemma \ref{L3}]
		In view of the marginal conditions \eqref{marginal}, we write and estimate
		\[
			\int_0^T \iint \Ub(t,x) + \Ub(t,y)\, d\pi_t(x,y) =\int_0^T \int \Ub |\rho|\, dxdt\leq \|\rho\|_{L^\infty(L^2)} \|\Ub\|_{L^1(L^2)},
		\]
		where we have used the Cauchy-Schwarz inequality, knowing that $\rho\in L^{\infty}((0,T);L^2(\R^n))$ via interpolation of norms. It remains to apply \eqref{11}.																		  
	\end{proof}
	To conclude the proof of Proposition \ref{pr3}, we substitute \eqref{15} into the estimate of Lemma \ref{L1} and apply Lemmas \ref{L2} and \ref{L3}.

	\section{Vanishing viscosity for 2D Euler equation}\label{sec:euler}
																									
	In this section we are going to exploit the uniqueness result for the linear equation to prove the second main theorem of the paper, namely Theorem \ref{Thm2}. The Cauchy problem for the two-dimensional Euler equations in vorticity formulation in $(0,T)\times\R^2$ is the following 
	\begin{equation}\label{eq:E}
		\left\{\begin{array}{rcll}
		\partial_t\omega+u\cdot\nabla\omega \!\!\!& = \!\!\!& 0          &  \mbox{in }(0,T)\times \R^2\,,\\
		u                                 \!\!\!  & =\!\!\! & k\ast \omega  &  \mbox{in }(0,T)\times \R^2\,,\\
		\omega|_{t=0}                    \!\!\!   & = \!\!\!& \omega_0 &  \mbox{in } \R^2\,,
		\end{array}      \right.     
	\end{equation}
	where we recall that the vorticity $\omega\in\R$ and the velocity field $u\in\R^2$ are unknown and  $k$ is the Biot-Savart kernel
	\begin{equation*}
		k(x)=\frac{1}{2\pi}\frac{x^{\perp}}{|x|^2}.
	\end{equation*}
	The initial datum is assumed to satisfy 
	\begin{equation}\label{eq:ide}
		\omega_0\in L_{c}^{1}(\R^2)\cap H^{-1}_{loc}(\R^2),
	\end{equation}
	where $L^1_c$ denotes the spaces of compactly supported integrable functions and  $\omega_0\in H^{-1}_{loc}$ means that $\psi\omega_0\in H^{-1}$ for any $\psi\in C^{\infty}_c$. This condition is important in the following as it provides a local bound on the kinetic energy for the velocity $u_0$.
	We recall, see \cite{BBC}, that if $\omega\in L^{\infty}((0,T);L^{1}(\R^2))$, the velocity field $u=k\ast \omega$ is in the class of velocity fields considered in Section \ref{sec:pre}. Indeed, the Biot-Savart kernel $k$ has distributional derivative given by 
	\begin{equation*}
		\partial_{j}k(x)=\frac{1}{2\pi}\partial_{j}\left(\frac{-x_2}{|x|^2},\frac{x_1}{|x|^2}\right)
	\end{equation*}
	and its Fourier transform $\widehat{\partial_{j}k^{i}}$ is bounded in $L^{\infty}(\R^2)$. %Then, it is well-known that the convolution operators $S^{i}_{j}$ defined by $S^{i}_{j}\omega=\partial_{j}k^{i}\ast \omega$ for every Schwartz function $\omega$ extend to mappings from $L^{1}(\R^2)$ to $L^{1,\infty}(\R^2)$ with 
	%\begin{equation*}
	%	\|S^{i}_{j}\omega\|_{L^{1,\infty}(\R^n)}\lesssim \|\omega\|_{L^1}.
	%\end{equation*}
	%We refer to \cite{BouchutCrippa13}, where the extension in the case $p=1$ is explained precisely.\par

	For any $\nu>0$, the Cauchy problem for the two dimensional Navier-Stokes equations in vorticity formulation is given by
	\begin{equation}\label{eq:NS}
		\left\{\begin{array}{rcll}
		\partial_t\omv+\uv\cdot\nabla\omv \!\!\!& = \!\!\!& \nu\Delta\omv & \mbox{in }(0,T)\times \R^2\,,\\
		\uv                              \!\!\! & = \!\!\!& k\ast\omv   & \mbox{in }(0,T)\times \R^2\,,  \\
		\omv|_{t=0}                    \!\!\!   & = \!\!\!& \omv_0 &\mbox{in }  \R^2\,,       
		\end{array}\right.
	\end{equation}
	where the vorticity $\omv\in\R$ and the velocity field $\uv\in\R^2$ are unknown. We assume that $\{\omv_0\}_{\nu}$ are supported in the same compact set and satisfy the following hypotheses:
	\begin{equation}\label{eq:idn}
		\left\{\begin{aligned}
		& \omv_{0}\in C^{\infty}_{c}(\R^2)\,,                                              \\
		& \{\omv_{0}\}_{\nu}\subset L^{1}(\R^2)\cap H^{-1}_{loc}(\R^{2}) \quad \textrm{ uniformly}\,, \\
		& \omv_{0}\rightarrow \omega_{0}\textrm{ in }L^{1}(\R^2)\,.                            
		\end{aligned}\right.
	\end{equation}
	We note that given $\omega_0\in L_{c}^{1}(\R^2)\cap H^{-1}_{loc}(\R^2)$ it is easy to construct (e.g., by convolution) a sequence $ \{\omv_{0}\}_{\nu}$ satisfying \eqref{eq:idn}. Finally, we recall, see \cite[Theorem 3.2A]{MajdaBertozzi02}, that given $\omv_{0}$ satisfying \eqref{eq:idn} there exists a unique smooth solution $(\uv,\omv)$ of the Cauchy problem \eqref{eq:NS}. The main goal of this section is to prove that, up to subsequences, the limit of  the sequence $(\uv,\omv)$ exists and satisfies the Euler equations \eqref{eq:E} as a Lagrangian and renormalized solution. For the definitions of Lagrangian and renormalized solutions we refer to the linear case, i.e., Definitions \ref{def:rs} and \ref{def:ls} above, which have both to be augmented by the Biot--Savart condition $u = k\ast \omega$. Notice that $JX\equiv1$ thanks to the incompressibility condition $\div u=0$.
						
	%The 
	%				
	%				
	%				Precise definitions, which are in fact analogous to the ones in the linear setting, follow, see Definition \ref{def:ls} below. Precisely, with Lagrangian solutions of the Euler equations we mean solutions in the following sense.
	%				\begin{definition}\label{def:lse}
	%					Let $\omega\in C([0,T);L^{1}(\R^2))$. The pair $(u,\omega)$ is a Lagrangian solutions of \eqref{eq:E} if 
	%						\begin{equation*}
	%							\begin{aligned}
	%								  & \omega=\omega_{0}(X^{-1}(t,\cdot)(x))\textrm{ for all $t\in(0,T)$ and a.e. $x\in\R^2$}, \\
	%								  & u=K\ast\omega\textrm{  a.e. in }(0,T)\times\R^2\quad \mbox{ and }                       \\
	%								  & \textrm{$X$ is a regular Lagrangian flow associated to }u.                              
	%							\end{aligned}
	%						\end{equation*}
	%					\end{definition}
	For the convenience of the reader we rewrite the statement of Theorem \ref{Thm2} in a more detailed form. 
	\begin{theorem}\label{teo:4}
		Let $\om_0$ and $\{\omv_0\}_{\nu}$ satisfying \eqref{eq:ide} and \eqref{eq:idn}. Let $(\uv, \omv)$ be the unique smooth solution of \eqref{eq:NS}. Then, there exists $$(u,\om)\in  L^{\infty}((0,T);L_{loc}^{2}\cap L^{2,\infty}(\R^{2}))\times L^{\infty}((0,T);L^{1}(\R^2))$$ such that, up to subsequences, for any $1\leq p<2$
		\begin{equation}\label{eq:conv1}
			\left\{\begin{aligned}
			& \uv\rightarrow u\textrm{ strongly in }L^{p}((0,T);L^{p}_{loc}(\R^2))                     \\
			& \omv\stackrel{*}{\rightharpoonup} \om\textrm{ weakly* in }L^{\infty}((0,T);L^{1}(\R^2))\,. 
			\end{aligned}\right.
		\end{equation}
		Moreover, $(u,\omega)$ satisfies the Euler equations in the sense of Lagrangian and renormalized solutions.
	\end{theorem}
	In the above theorem the weak* convergence in $L^{\infty}((0,T);L^{1}(\R^2))$ is intended in the duality with $L^{1}((0,T);L^{\infty}(\R^2))$.
	We divide the proof of Theorem \ref{teo:4} in several lemmas. In the first lemma we prove the compactness result stated in Theorem \ref{teo:4}.
	\begin{lemma}\label{lem:1}
		Under the hypothesis of Theorem \ref{teo:4} there exists $$(u,\om)\in  L^{\infty}((0,T);L_{loc}^{2}\cap L^{2,\infty}(\R^{2}))\times L^{\infty}((0,T);L^{1}(\R^2))$$ such that the convergences \eqref{eq:conv1} hold.
		%						\begin{equation}\label{eq:conv1}
		%							\begin{aligned}
		%								  & \uv\rightarrow u\textrm{ strongly in }L^{p}((0,T);L^{p}_{loc}(\R^2))                     \\
		%								  & \omv\stackrel{*}{\rightharpoonup} \om\textrm{ weakly* in }L^{\infty}((0,T);L^{1}(\R^2)). 
		%							\end{aligned}
		%						\end{equation}
	\end{lemma}
	\begin{proof}
					                
		By the uniform bounds on the initial datum \eqref{eq:idn} and the global existence of smooth solutions of the two-dimensional Navier-Stokes equations it follows that, see DiPerna and Majda \cite[Section 2A]{DPM1}, for any compact $Q\subset\R^2$ and some $s>0$ the following uniform bounds hold:
		\begin{equation}\label{eq:ub1}
			\begin{aligned}
				  & \sup_{\nu}\sup_{t\in(0,T)}\int_{Q}|\uv|^2\, dx \le C(Q),\qquad \sup_{\nu}\sup_{t\in(0,T)}\int|\omv|\leq C, \\
				\\
				  & \{\uv\}_{\nu}\subset \textrm{Lip}((0,T); H_{loc}^{-s}(\R^2)). 
			\end{aligned}
		\end{equation} 
		Then, by standard weak compactness arguments, see \cite[Theorem 1.1]{DPM1}, there exists $(u,\om)\in L^{\infty}(0,T,L^{2}_{loc}(\R^2))\times L^{\infty}((0,T);\mathcal{M}(\R^2))$, where $\mathcal{M}(\R^2)$ is the space of Radon measures, such that up to a subsequence, not relabeled, the following convergences hold,
		\begin{equation*}
			\left\{\begin{aligned}
			& \uv\rightarrow u\textrm{ strongly in }L^{p}((0,T);L^{p}_{loc}(\R^2))\,, \quad 1\leq p<2,                           \\
			& \omv\stackrel{*}{\rightharpoonup} \om\textrm{ weakly* in }L^{\infty}((0,T);\mathcal{M}(\R^2)). 
			\end{aligned}\right.
		\end{equation*}
		Moreover, by the weak Hardy-Littlewood-Sobolev inequality, see \cite[Lemma 4.5.7]{H} it holds
		\begin{equation}\label{eq:ub3}
			\|\uv\|_{L^{\infty}((0,T);L^{2,\infty}(\R^2))}\leq c\|\omv\|_{L^{\infty}((0,T);L^1(\R^2))},
		\end{equation}
		and this implies $u\in  L^{\infty}((0,T); L^{2,\infty}(\R^{2}))$.
		To prove that $\omega\in L^{\infty}((0,T);L^{1}(\R^2))$, by Dunford--Pettis theorem we just need to prove that the sequence $\{\omv\}_{\nu}$ is equi-integrable in space. We start by noticing that since the sequence of initial data $\{\omv_0\}_{\nu}$ is strongly convergent in $L^{1}_{c}(\R^2)$ there exists $G:[0,\infty]\to[0,\infty]$ such that $G\in C^{1}(\R)$, $G(0)=0$, $G$ is convex and increasing and 
		\begin{equation}\label{eq:equi1}
			\begin{aligned}
				  & \lim_{s\to\infty}\frac{G(s)}{s}=\infty\qquad \mbox{and}\qquad \sup_{\nu} & \int_{\R^2}G(|\omv_0(x)|)\,dx<\infty. 
			\end{aligned}
		\end{equation}
		Then, by a suitable truncation argument, it follows that $\omv$ satisfies
		\begin{equation}\label{eq:equi2}
			\partial_{t}|\omv|-\nu\Delta|\omv|+\uv\cdot\nabla|\omv|\leq 0.
		\end{equation}
		Multiplying \eqref{eq:equi2} by $G'(|\omv|)$ and integrating by parts and using the divergence-free condition we get 
		\begin{equation}\label{eq:equi3}
			\frac{d}{dt}\int_{\R^2}G(|\omv(t,x)|)\,dx+\nu\int_{\R^2}G''(|\omv(t,x)|)|\nabla|\omv||^2\,dx\leq 0. 
		\end{equation}
		Using that $G$ is convex, integrating in time and exploiting \eqref{eq:equi1}, there exists a constant $C>0$ independent on the viscosity $\nu$ such that 
		\begin{equation}\label{eq:equi4}
			\sup_{t\in(0,T)}\int_{\R^2}G(|\omv(t,x)|)\leq C. 
		\end{equation}
		In turn this implies that given any $\e>0$ there exists $\delta=\delta(\e)$ such that, for any measurable set $A\subset\R^2$ such that $\mathcal{L}^2(A)\leq \delta$, it holds
		\begin{equation}\label{eq:equi5}
			\sup_{t\in(0,T)}\int_{A}|\omv|\,dx\leq \e. 
		\end{equation}
		In order to conclude that $\omega\in L^{\infty}((0,T);L^{1}(\R^2))$ we need to prove that $\{\omv\}_{\nu}$ is equi-integrable at infinity. We start by noting that by \eqref{eq:idn} there exists a radius $\tilde{R}=\tilde{R}(\|\omega_0\|_{1})$ such that 
		\begin{equation}\label{eq:equi6}
			\int_{|x|>\tilde{R}}|\omv_0|\,dx=0.
		\end{equation}
		Let now $r$ and $R$ be such that $\tilde{R}<r<R /2$. Let  $\phi_{r}^{R}\in C^{\infty}_{c}(\R^{2})$ be the cut-off function defined as																								
		$$\psi_{r}^{R}(x)=\begin{cases} 0 &\mbox{if }|x|\in [0,r]\,,\\
		1&\mbox{if }|x|\in [2r,R]\,,\\
		0 &\mbox{if }|x|\in [2R,\infty)\,.
			\end{cases}$$
			Then, 
			\begin{equation}\label{eq:equi7}
				\begin{aligned}
					  & |\nabla \psi_{r}^{R}|\leq \frac{C}{r},\quad & |\nabla^{2}\psi_{r}^{R}|\leq\frac{C}{r^2}. 
				\end{aligned}
			\end{equation}
			Let $\beta\in C^{1}(\R)\cap L^{\infty}(\R)$ be a convex function, then by \eqref{eq:NS} we have
			\begin{equation}\label{eq:equi8}
				\partial_{t}\beta(\omv)+\uv\cdot\nabla\beta(\omv)-\nu\Delta\beta(\omv)+\nu\beta''(\omv)|\nabla\omv|^2=0.
			\end{equation}
			By multiplying \eqref{eq:equi8} by $\psi_{r}^{R}$, integrating by parts, integrating in time and using \eqref{eq:equi6} we get for all $t\in(0,T)$
			\begin{equation}\label{eq:equi9}
				\int\beta(\omv)\psi_{r}^{R}\,dx\leq \iint |\uv||\beta(\omv)||\nabla\psi_{r}^{R}|\,dxdt+\nu\iint |\beta(\omv)||\Delta\psi_{r}^{R}|\,dxdt.
			\end{equation}
			Let $M>0$. By a simple approximation argument we can choose $\beta(s)=|s|\wedge M$. Then, after sending $R\to\infty$ in \eqref{eq:equi9} and using \eqref{eq:equi7} we have 
			\begin{equation}\label{eq:equi10}
				\begin{aligned}
					\int_{\{|x|>2r\}}(|\omv|\wedge M)\,dx & \leq\frac{1}{r}\iint|\uv|(|\omv|\wedge M)\,dxdt   +\frac{\nu}{r^2}\iint|\omv|\wedge M\,dxdt. 
				\end{aligned}
			\end{equation}
			A simple manipulation leads to the following inequality for all $t\in(0,T)$
			\begin{equation}\label{eq:equi11}
				\begin{aligned}
					\int_{\{|x|>2r\}}|\omv|\,dx & \leq \int_{\{|\omv|>M\}}|\omv|\,dx                                                           \\
					                            & \leq\frac{1}{r}\iint|\uv|(|\omv|\wedge M)\,dxdt   +\frac{\nu}{r^2}\iint|\omv|\wedge M\,dxdt. 
				\end{aligned}
			\end{equation}
			Let us now decompose the kernel $k= k_1+k_2$, where $k_{1}=k\chi_{B_{1}(0)}\in L^{1}(\R^2)$ and $k_{2}= k\chi_{B_{1}(0)^{c}}\in L^{\infty}(\R^2)$. The decomposition of the kernel induces the decomposition $\uv=\uv_{1}+\uv_{2}$  and, by Young's inequality (for convolution), we have  
			the uniform bounds
			\begin{equation}\label{eq:ub2}
				\{\uv_1\}_{\nu}\in L^{\infty}((0,T);L^{1}(\R^2))\,,\qquad
				\{\uv_2\}_{\nu}\in L^{\infty}((0,T)\times\R^2) \,.  
			\end{equation}
			Using the above decomposition we infer that for all $t\in(0,T)$ 
			\begin{equation}\label{eq:equi12}
				\begin{aligned}
					\int_{\{|x|>2r\}}|\omv|\,dx & \leq\sup_{t\in(0,T)}\int_{\{|\omv|>M\}}|\omv|\,dx +\frac{MT}{r}\sup_{t\in(0,T)}\|\uv_1\|_{L^1(\R^2)}                                       \\
					                            & \quad +\frac{T}{r}\sup_{t\in(0,T)}\left(\|\uv_2\|_{L^2(\R^2)}\|\omv\|_{L^1(\R^2)}\right)   +\frac{\nu}{r^2}\|\omega^{\nu}_0\|_{L^1(\R^2)} \\
					                            & =(I)+(II)+(III)+(IV).                                                                                                                   
				\end{aligned}
			\end{equation}
			We are now going to estimate all the term separately: First, we note that for any $t\in(0,T)$ we have that 
			\begin{equation}\label{eq:equi13}
				\begin{aligned}
					\mathcal{L}^2(\{x\in\R^2:|\omv(t,x)|>M\} & \leq \frac{1}{M}\|\omv\|_{L^1(\R^2)}     \\
					                                         & \leq \frac{1}{M}\|\omv_0\|_{L^1(\R^2)}        \\
					                                         & \leq \frac{C}{M}\|\omega_{0}\|_{L^1(\R^2)}. 
				\end{aligned}
			\end{equation}
			Let $\e>0$ and $\delta=\delta(\e)$ given in \eqref{eq:equi5}, then we can choose $M=M(\e)$, independent of the time, such that
			\begin{equation*}
				\mathcal{L}^2(\{x\in\R^2:|\omv(t,x)|>M\})\leq \delta,\quad\textrm{ for any }t\in(0,T).
			\end{equation*}                                                  
			Then, by \eqref{eq:equi5}, 
			\begin{equation*}
				\sup_{t\in(0,T)}\int_{\{|\omv|>M\}}|\omv|\,dx\leq\frac{\e}{4}.
			\end{equation*}
			With this choice of $M$ fixed, since we can assume without loss of generality $\nu<1$, we can infer that there exists $r=r(\e)$ such that 
			\[
				(II)\leq\frac{\e}{4},\quad  (III)\leq\frac{\e}{4},\quad  (IV)\leq\frac{\e}{4}. 
			\]
			Then, we have just proved that for any $\e>0$ there exists $r=r(\e)$ such that 
			\begin{equation*}
				\sup_{t\in(0,T)}\int_{\{|x|>2r\}}|\omv|\,dx\leq \e.
			\end{equation*}
			This together with \eqref{eq:equi5} implies 						
			\begin{equation*}
			%\label{eq:conv2}
			\omv\stackrel{*}{\rightharpoonup} \om\textrm{ in }L^{\infty}((0,T);L^{1}(\R^2)).\qedhere
			\end{equation*}
		\end{proof}
		In the following lemma we prove a duality formula for the limit $(u,\omega)$ obtained in Lemma~\ref{lem:1}. 
		\begin{lemma}\label{lem:2}
			Let $(u,\omega)$ be as in Lemma~\ref{lem:1}. Then, for any $\chi\in C^{\infty}_{c}((0,T)\times\R^2)$ there exists $\phi_{1}\in L^{\infty}((0,T);L^{1}\cap L^{\infty}(\R^{2}))$ solving in the sense of distributions 
			\begin{eqnarray}\label{eq:BP3}\left\{\begin{array}{rcll}
				-\partial_t\phi_1-\div(u\phi_1) \!\!\! & =\!\!\! & \chi    &\mbox{in }(0,T)\times\R^2\,,   \\
				\phi_{1}|_{t=T}               \!\!\!  & =\!\!\! & 0      &\mbox{in } \R^2\,,    
				\end{array}\right.
			\end{eqnarray}
			with $u= k\ast\omega$. Moreover, the following duality formula holds 
			\begin{equation}\label{eq:BP4}
				\iint\chi\omega\,dxdt=\int\omega_0\phi_{1}|_{t=0}\,dx.
			\end{equation}
		\end{lemma}
		\begin{proof} First we prove that $u=k\ast\omega$ a.e.~in $(0,T)\times\R^{2}$. Let $\eta\in C^{\infty}_{c}((0,T)\times\R^{2})$, then 
			\begin{equation}\label{eq:bs1}
				\begin{aligned}
					0=\lim_{\nu\to 0}\iint(\uv-(k\ast\omv))\eta\,dxdt & =\lim_{\nu\to 0}\iint\uv\eta-\omv(k\ast\eta)\,dxdt \\
					                                                  & =\iint u\eta-\omega(k\ast\eta)\,dxdt               \\
					                                                  & =\iint (u-(k\ast\omega))\eta\,dxdt,                 
				\end{aligned}
			\end{equation}
			where the convergences \eqref{eq:conv1} have been used together with the bound $k\ast\phi\in L^{\infty}((0,T)\times\R^{2})$, which holds because $\eta\in C^{\infty}_{c}((0,T)\times\R^{2})$.\par					Let $\chi\in C^{\infty}_{c}((0,T)\times\R^2)$ and let $\phiv$ be the unique smooth solution of the following Cauchy problem
			\begin{eqnarray}\label{eq:BP1}
				\left\{\begin{array}{rcll}
				-\partial_t\phiv-\nu\Delta\phiv-\div(\uv\phiv) & =&\chi & \quad \mbox{ in } (0,T)\times\R^2 \,,\\
				\phiv|_{t=T}                                   & =& 0   & \quad \mbox{ in } \R^2\,,
				\end{array}
				\right.
			\end{eqnarray}
			where $\{\uv\}_{\nu}$ is the subsequence chosen in Lemma \ref{lem:1}. 
			Then, by multiplying \eqref{eq:NS} by $\phiv$ and integrating by parts we get 
			\begin{equation}\label{eq:BP2}
				\iint\omv\chi\,dxdt=\int\omv_0\phiv|_{t=0}\,dx.
			\end{equation}
			Since $\chi\in C^{\infty}_{c}((0,T)\times\R^2)$, by standard energy estimates
			it follows that 
			\begin{equation}\label{eq:ub4}
				\left\{\begin{aligned}
					  & \{\phiv\}_{\nu}\subset L^{\infty}((0,T);L^{\infty}\cap L^{1}(\R^2))\,, \\
					  & \{\sqrt{\nu}\nabla\phiv\}_{\nu}\subset L^{2}((0,T)\times\R^2)\,,       
				\end{aligned}\right.
			\end{equation}
			with uniform bounds. Then up to subsequences there exists $\phi_{1}\in L^{\infty}((0,T);L^{\infty}\cap L^{1}(\R^2))$ such that 
			\begin{equation}\label{eq:conv3}
				\phiv\stackrel{*}{\rightharpoonup}\phi_1\textrm{ in }L^{\infty}((0,T);L^{\infty}\cap L^{1}(\R^2)).
			\end{equation}
			Moreover, by using \eqref{eq:BP1} the convergence in \eqref{eq:conv3} can be upgraded to 
			\begin{equation}\label{eq:conv4}
				\phiv\rightarrow\phi_1\textrm{ in }C((0,T); L_{w^*}^{\infty}(\R^2)),
			\end{equation}
			where $C((0,T); L_{w^*}^{\infty}(\R^2))$ is the space of continuous functions with value in $L^{\infty}(\R^{2})$ endowed with the weak$^*$ topology. Then, by \eqref{eq:conv1}, \eqref{eq:BP2}, \eqref{eq:ub4} and \eqref{eq:conv3} it follows that $\phi_1$ is a distributional solution of the backward transport Cauchy problem \eqref{eq:BP3} and 
			%						\begin{equation}\label{eq:BP3}
			%							\begin{array}{rll}
			%								-\partial_t\phi_1-\div(u\phi_1) & = & \chi        \\
			%								u                               & = & K\ast\omega \\
			%								\phi_{1}|_{t=T}                 & = & 0           
			%							\end{array}
			%						\end{equation}
			%						and the following duality formula holds 
			%						\begin{equation}\label{eq:BP4}
			%							\iint\chi\omega\,dxdt=\int\omega_0\phi_{1}|_{t=0}\,dx,
			%						\end{equation}
			\eqref{eq:BP4} holds. 
		\end{proof}
		Now, we are in position to prove Theorem \ref{teo:4}. 
		\begin{proof}[Proof of Theorem \ref{teo:4}]
			Let us consider the Cauchy problem for the following linear continuity equation, 
			\begin{equation}\label{eq:LC1}
				\left\{\begin{array}{rlll}
				\partial_{t}w+\div(uw) & = & 0& \quad \mbox{ in } (0,T)\times \R^2\,,           \\
				w|_{t=0}               & = & \omega_0 & \quad \mbox{ in } \R^2 \,,   
				\end{array}\right.
			\end{equation}
			with $u=k\ast\omega$.
			We regularize the velocity field and the initial datum by using classical mollification. Then, we obtain sequences $\{u_{\delta}\}_{\delta}$ and $\{\omega_{0,\delta}\}_{\delta}$. By the standard 
			Cauchy-Lipschitz theory we can find a sequence of smooth functions $\{w^{\delta}\}_{\delta}$ uniformly bounded in $ L^{\infty}((0,T);L^{1}(\R^2))$ such that 
			\begin{equation}\label{eq:LC2}
				\left\{\begin{array}{rlll}
				\partial_{t}w^{\delta}+\div(u_{\delta}w^{\delta}) & = & 0   &\quad \mbox{ in } (0,T)\times \R^2\,,            \\
				w^{\delta}|_{t=0}                                 & = & \omega_{0,\delta} &\quad \mbox{ in } \R^2\,. 
				\end{array}\right.
			\end{equation}
			It is proved in \cite{BouchutCrippa13} that Lagrangian solutions are stable under smooth approximation. Therefore, there exists $\bar{w}\in C((0,T);L^{1}(\R^2))$ such that 
			\begin{equation*}
				w^{\delta}\rightarrow \bar{w}\textrm{ in }C((0,T);L^{1}(\R^2))
			\end{equation*}
			and $\bar{w}$ is a Lagrangian solution of \eqref{eq:LC1} in the sense of Definition \ref{def:ls}. Note that this is the first crucial point in this section where we really need to use the fact that $\omega\in L^{\infty}((0,T);L^{1}(\R^2))$. Finally, let $\chi\in C^{\infty}_{c}((0,T)\times\R^2)$ and let $\phi^{\delta}$ be the unique smooth solution of the following backward transport Cauchy problem
			\begin{equation}\label{eq:BT1}
				\left\{\begin{array}{rlll}
				-\partial_t\phi^{\delta}-\div(u_{\delta}\phi^{\delta}) & = & \chi  &\quad \mbox{ in } (0,T)\times \R^2\,,\\
				\phi^{\delta}|_{t=T}                                   & = & 0 &\quad \mbox{ in } \R^2\,.  
				\end{array}\right.
			\end{equation}
			Arguing as in Lemma \ref{lem:2}, we can infer that there exists $\phi_{2}\in L^{\infty}((0,T);L^{\infty}\cap L^{1}(\R^2))$ distributional solution of 
			\begin{eqnarray}\label{eq:BT2}
				\left\{\begin{array}{rclll}
				-\partial_t\phi_2-\div(u\phi_2) & =\chi  & \quad  \mbox{ in } (0,T)\times \R^2\,,   \\
				\phi_{2}|_{t=T}                 & =0 & \quad \mbox{ in } \R^2\,,          
				\end{array}\right.
			\end{eqnarray}
			with $u=k\ast\omega$ and 
			\begin{equation}\label{eq:BT3}
				\iint\chi\bar{w}\,dxdt=\int\omega_0\phi_{2}|_{t=0}\,dx.
			\end{equation}
			We need to prove that $\phi_1=\phi_2$. By subtracting \eqref{eq:BP3} and \eqref{eq:BT2} we have that the difference $\phi_1-\phi_{2}$ is a distributional solution in $L^{\infty}(0,T;L^{1}\cap L^{\infty}(\R^{2}))$ of the Cauchy problem 
			\begin{equation}\label{eq:main}
				\left\{\begin{array}{rlll}
				\partial_t(\phi_1-\phi_{2})+\div(u(\phi_1-\phi_{2})) & = & 0  & \quad  \mbox{ in } (0,T)\times \R^2\,,  \\
				(\phi_1-\phi_{2})|_{t=T}                             & = & 0 & \quad \mbox{ in } \R^2\,.
				\end{array}\right.
			\end{equation}
			Since the velocity field $u$ satisfies K1)--K5) and \eqref{eq:ub3}, by Theorem \ref{th1} it follows that $\phi_{1}=\phi_{2}$ a.e. in $(0,T)\times\R^2$.\par We are ready to conclude: Subtracting \eqref{eq:BP4} and \eqref{eq:BT3} we have that 
			\begin{equation*}
				\iint\chi(\omega-\bar{w})\,dxdt=0\,
			\end{equation*}
			and by varying $\chi\in C^{\infty}_{c}((0,T)\times\R^2)$ we obtain $\omega=\bar{w}$ a.e. in $(0,T)\times\R^2$ and then $\omega$ is Lagrangian and renormalized. 
		\end{proof}

		\subsection*{Acknowledgement}
		G.~Crippa is partially supported by the Swiss National Science Foundation grant 200020\_156112 and by the ERC Starting Grant 676675 FLIRT.\break S.~Spirito acknowledges the support of INdAM-GNAMPA. Part of this work has been done when C.~Nobili was visiting researcher at the Collaborative Research Center (CRC) of the University of Bonn.

		\bibliography{transport}
		\bibliographystyle{acm}

		\end{document}